\documentclass[a4paper,12pt]{amsart}

\usepackage{times}
\usepackage{amsfonts}
\usepackage[hmargin=30mm, vmargin=25mm, includefoot, twoside]{geometry}
 \usepackage{graphicx,color,amssymb, amsmath, amsthm, mathrsfs}
\usepackage{amsmath,amscd}
\usepackage{geometry}
\usepackage[all]{xy}
\usepackage{slashed}
\usepackage{mathtools}
\usepackage{todonotes}
\usepackage{subcaption}
\usepackage{color}
\usepackage[colorlinks = true,
            linkcolor = blue,
            urlcolor  = blue,
            citecolor = red,
            anchorcolor = blue]{hyperref}
\usepackage{tikz}
\usepackage{mwe}

\usepackage[linewidth=1pt]{mdframed}
\usepackage{lipsum}

\swapnumbers
\newtheorem{theorem}{Theorem}[section]
\newtheorem{lemma}[theorem]{Lemma}
\newtheorem{proposition}[theorem]{Proposition}

\newtheorem{conjecture}[theorem]{Conjecture}
\theoremstyle{definition}
\newtheorem{definition}[theorem]{Definition}

\newtheorem{remark}[theorem]{Remark}

\newtheorem{example}[theorem]{Example}
\newtheorem*{remark*}{Remark}
\newtheorem*{definition*}{Definition}

\usepackage{thmtools}
\declaretheoremstyle[notefont=\bfseries,notebraces={}{},%
    headpunct={},postheadspace=1em]{mystyle}
\declaretheorem[style=mystyle,numbered=no,name=Theorem]{thm-hand}

\newcommand{\bea}          {\begin{eqnarray}}
\newcommand{\eea}          {\end{eqnarray}}
\newcommand{\beastar}          {\begin{eqnarray*}}
\newcommand{\eeastar}          {\end{eqnarray*}}

\newcommand{\prop}{\textnormal{prop }}
\newcommand{\supp}{\textnormal{supp }}
\newcommand{\asdim}{\textnormal{asdim }}
\newcommand{\Diam}{\textnormal{Diam}}
\newcommand{\topo}[1]

\begin{document}

\title{$L^p$ coarse Baum-Connes conjecture via $C_{0}$ coarse geometry}

\author{Hang Wang}
\address[Hang Wang]{\normalfont{Research Center for Operator Algebras, School of Mathematical Sciences, East China Normal University, Shanghai 200062, China}}
\email{wanghang@math.ecnu.edu.cn}
\thanks{Hang Wang is supported by NSFC 12271165 and in part by the Science and Technology Commission of Shanghai Municipality (No. 22DZ2229014, 23JC1401900)}

\author{Yanru Wang}
\address[Yanru Wang]{\normalfont{Tianyuan Mathematical Center in Southwest China, School of Mathematics\\ Sichuan University, Chengdu 610065, China}}
\email{52215500014@stu.ecnu.edu.cn}
\thanks{Yanru Wang is supported by the Science and Technology Commission of Shanghai Municipality (No. 23JC1401900). Yanru Wang is the corresponding author.}

\author{Jianguo Zhang}
\address[Jianguo Zhang]{\normalfont{School of Mathematics and Statistics, Shaanxi Normal University,
Xi'an 710119, China}}
\email{jgzhang@snnu.edu.cn}
\thanks{Jianguo Zhang is supported by NSFC 12271165, 12171156, and 12301154}

\author{Dapeng Zhou}
\address[Dapeng Zhou]{\normalfont{School of Statistics and Information, Shanghai University of International Business and Economics, Shanghai 201620, China}}
\email{giantroczhou@126.com}
\thanks{Dapeng Zhou is supported by NSFC 12271165}

\date{\today}

\begin{abstract} 
In this paper, we investigate the $L^{p}$ coarse Baum-Connes conjecture for $p\in [1,\infty)$ via $C_{0}$ coarse structure, which is a refinement of the bounded coarse structure on a metric space. We prove that the $C_{0}$ version of the $L^{p}$ coarse Baum-Connes conjecture holds for a finite-dimensional simplicial complex equipped with a uniform spherical metric. Using this result, we construct an obstruction group for the $L^{p}$ coarse Baum-Connes conjecture. As an application, we show that the obstruction group vanishes under the assumption of finite asymptotic dimension, thereby providing a new proof of the $L^{p}$ coarse Baum-Connes conjecture in this case.
\end{abstract}

\maketitle
\tableofcontents

\section{Introduction} 
The coarse Baum-Connes conjecture for metric spaces formulated by J. Roe in \cite{Roe93} is a coarse analog of the celebrated Baum-Connes conjecture for groups first proposed by P. Baum and A. Connes in \cite{BC00}, and reformulated in \cite{BCH94}: both bridging noncommutative geometry with classical topology and geometry in spirit. For a proper metric space $X$,  Roe established the coarse Baum-Connes assembly map in the article \cite{Roe93}. The left-hand side of this map, or the topological side, which involves the limit of the $K$-homology groups of the Rips complex of $X$, is local and computable, while the right-hand side, or the noncommutative side, which is the $K$-theory of a certain $C^{*}$-algebra, namely Roe algebra of $X$, is the receptor for higher indices of elliptic operators. If $X$ has bounded geometry, then the coarse Baum-Connes conjecture posits that the coarse Baum-Connes assembly map is an isomorphism. The coarse Baum-Connes conjecture implies the injectivity part of the Baum-Connes conjecture, which deduces the Novikov conjecture, when $X$ is a finitely generated group equipped with a word metric. In addition, the coarse Baum-Connes conjecture has several applications, such as Gromov's zero-in-the-spectrum conjecture and the positive scalar curvature conjecture, when $X=M$, a complete Riemannian manifold \cite{Yu97A}.

In the past three decades, a great deal of effort has been devoted to the study of the coarse Baum-Connes conjecture. Some researchers showed that certain expander graphs are counterexamples to the coarse Baum-Connes conjecture \cite{HLS02, WY12}. On the other side, the coarse Baum-Connes conjecture has been verified for a variety of spaces. At first, in \cite{HR95}, N. Higson and J. Roe verified the coarse Baum-Connes conjecture for non-positively curved spaces, including affine buildings and Gromov hyperbolic metric spaces. Later, G. Yu proved that it holds for proper metric spaces with finite asymptotic dimension \cite{Yu98}, then more generally, for discrete metric spaces that admit a uniform embedding into Hilbert space \cite{Yu00}. Next, T. Fukaya and S. Oguni found that it is true for certain relatively hyperbolic groups \cite{FO12}, direct products of geodesic Gromov hyperbolic spaces and Busemann non-positively curved spaces \cite{FO15}, Busemann non-positively curved spaces \cite{FO16}, and proper coarsely convex spaces \cite{FO20}. Recently, J. Deng, Q. Wang, and G. Yu showed that the coarse Baum-Connes conjecture holds for group extensions with a `CE-by-CE' structure including relative expanders and the special box spaces of free groups \cite{DWY23}. 

Investigating the $L^{p}$ version of the Baum-Connes conjecture is both exciting and valuable. Firstly, there are some excellent articles and surveys on the $L^{p}$ version of the Baum-Connes conjecture. In Lafforgue's approach to the Baum-Connes conjecture \cite{Laf02}, he invented the Banach $KK$-theory to study many discrete groups with property (T), including Gromov's hyperbolic groups, and established the Baum-Connes conjecture with values in the $L^{1}$ group subalgebra of the reduced group $C^{*}$-algebra for such groups. In a remarkable development \cite{Yu05}, G. Yu demonstrated that for any hyperbolic group  $\Gamma$, there exists $p\in [2,\infty)$ such that $\Gamma$ admits a proper affine isometric action on an $\ell^{p}$ space. Therefore, G. Kasparov and G. Yu explored the $L^{p}$ Baum-Connes conjecture for $p\in (1,\infty)$ and showed that this conjecture holds for groups with proper isometric actions on $\ell^{p}$ spaces in their unpublished work. Moreover, Y. C. Chung in \cite{Chu21} defined a certain $L^{p}$ assembly map and showed that the $L^{p}$ Baum-Connes conjecture for $\Gamma$ with coefficient in $C(X)$ is true if there is an action of a countable discrete group $\Gamma$ on a compact Hausdorff space $X$ having finite dynamical complexity. In this paper, we will explore classes of spaces for which the $L^{p}$ analog of the coarse Baum-Connes conjecture holds.

Furthermore, there has been renewed interest in $L^{p}$ operator algebras. Due to Phillips' work \cite{Phi13}, it is known that the $K$-theory of $L^{p}$ analogs of Cuntz algebras is the same as that of $C^{*}$-algebras. This work has inspired researchers to investigate $L^{p}$ operator algebras that behave like $C^{*}$-algebras, including group $L^{p}$ operator algebras \cite{Gar21, GT15, GT16, GT19}, $L^{p}$ crossed products \cite{WWZZ24, WZ23} and groupoid $L^{p}$ operator algebras \cite{GL17}. There is also related work on $\ell^{p}$ Roe algebras \cite{CL18, CL21, LWZ19}.

Moreover, in recent years, some mathematicians made significant progress on the $L^{p}$ coarse Baum-Connes conjecture. In this important article \cite{ZZ21}, J. Zhang and D. Zhou proved that for $p\in [1,\infty)$, the $L^{p}$ coarse Baum-Connes conjecture holds for spaces with finite asymptotic dimension, and showed that on such spaces, the $K$-theory of $L^{p}$ Roe algebras is independent of $p\in(1,\infty)$. Next,  L. Shan and Q. Wang in \cite{SW21} verified the injectivity part of the $L^{p}$ coarse Baum-Connes conjecture for $p\in (1,\infty)$ when $X$ coarsely embeds into a simply connected complete Riemannian manifold with non-positive sectional curvature, also called a Hadamard manifold. Recently, Y. C. Chung and P. W. Nowak in \cite{CN23} demonstrated that the $L^{p}$ coarse Baum-Connes conjecture fails for expanders arising from residually finite hyperbolic groups when $p\in (1,\infty)$. 

To study the $L^{p}$ coarse Baum-Connes conjecture, an important step is to define the $L^{p}$ coarse Baum-Connes assembly map. We first describe the left-hand side of this map. To do this, we briefly recall from \cite{Roe93} that an anti-\v{C}ech sequence for a proper metric space $X$ is a sequence $\{\mathcal{U}_{i}\}^{\infty}_{i=1}$ of successively coarser open covers of $X$ satisfying certain conditions. Furthermore, G. Yu in \cite{Yu97B} introduced localization algebras and claimed that the $K$-homology can be identified with the $K$-theory of the localization algebra. Based on Yu's work, J. Zhang and D. Zhou in \cite{ZZ21} introduced the $L^{p}$ localization algebras, and defined the $L^{p}$ coarse $K$-homology $KX^{p}_{*}(X)$,  which is the limit of $K$-theory groups of $L^{p}$ localization algebras of an anti-\v{C}ech sequence of $X$. On the right-hand side, we also have $L^{p}$ analogs of groups, the $K$-theory groups of the $L^{p}$ Roe algebra of controlled and locally compact operators on an $L^{p}$-$X$-module. Analogous to the coarse Baum-Connes assembly map, the $L^{p}$ coarse Baum-Connes assembly map is defined as follows:
$$
\mu: KX_{*}^{p}(X)=\varinjlim\limits_{i}K_{*}(B^{p}_{L}(N_{\mathcal{U}_{i}})) \rightarrow K_{*}(B^{p}(X)),
$$
where $\mathcal{U}_{i}$ is any anti-\v{C}ech sequence of $X$. The following $L^{p}$ coarse Baum-Connes conjecture was elaborated in \cite{ZZ21}. 

\begin{conjecture}[The $L^{p}$ coarse Baum-Connes conjecture]
Let $X$ be a proper metric space with finite asymptotic dimension. For $p\in [1,\infty)$, the $L^{p}$ coarse Baum-Connes assembly map $\mu$ is an isomorphism.
\end{conjecture}

Recall that a coarse structure on a set $X$ is a family of subsets of $X\times X$, which encodes information about the large-scale features of the space. Say that a bounded coarse structure on a metric space $(X,d)$ is a coarse structure consisting of all sets $A\subseteq X\times X$ for which the restriction of the distance function $d: A\rightarrow \mathbb{R}^{+}$ is bounded. An alternative approach to Yu's theorem on the coarse Baum-Connes conjecture \cite{Yu98} is $C_{0}$ coarse geometry. The notion of $C_{0}$ coarse geometry was introduced by N. Wright in \cite{Wri03} to explore the obstruction of properly positive scalar curvature on open manifolds. This new coarse geometry is finer than the bounded coarse structure at infinity, and it carries different information about the large-scale structure of a metric space compared to the bounded coarse structure.  Using this coarse structure, he established the $C_{0}$ version of the coarse Baum-Connes conjecture. Further, he introduced the total coarsening space to coarsen the metric space into a simplicial complex and thus built a `geometric' obstruction group to the coarse Baum-Connes conjecture on the bounded coarse structure \cite{Wri05}. As an application, he showed that the obstructions vanish for bounded geometry metric spaces with finite asymptotic dimension, hence giving a new proof of the coarse Baum-Connes conjecture in \cite{Wri05}. 

Inspired by his work, we investigate the $L^{p}$ coarse Baum-Connes conjecture from the perspective of $C_{0}$ coarse geometry. Using the abstract homology uniqueness result, we establish the $C_{0}$ version of the $L^{p}$ coarse Baum-Connes conjecture. Combining this result with the $L^{p}$ version of the Eilenberg swindles, the $L^{p}$ version of the Mayer-Vietoris sequence, and the fusion coarse structure between the bounded coarse structure and the $C_{0}$ coarse structure, we can create an obstruction group for the $L^{p}$ coarse Baum-Connes conjecture on the bounded structure. However, calculating obstructions directly can be quite challenging. To make this process easier, we introduce the hybrid coarse structure defined in \cite{Wri05}, which is more computable than the fusion coarse structure. By effectively utilizing this hybrid coarse structure, we demonstrate that obstruction groups arising from the fusion coarse structure vanish in bounded geometry metric spaces that have finite asymptotic dimension. Hence, we obtain a new proof of the $L^{p}$ coarse Baum-Connes conjecture for such spaces, which was first proved by J. Zhang and D. Zhou. The main contributions of this article are as follows: 

\begin{theorem}(see Theorem \ref{th 3.18})
Let $X$ be a finite-dimensional simplicial complex equipped with a uniform spherical metric, and let $X_{0}$ be the space $X$ with the $C_{0}$ coarse structure. Let $p\in [1,\infty)$, then 
$$
K^{p,\infty}_{*}(X)\cong KX^{p}_{*}(X_{0})\underset{\cong}{\xrightarrow{\mu}}K_{*}(B^{p}(X_{0})).
$$
\end{theorem}

This is the $L^{p}$ coarse Baum-Connes conjecture for $C_{0}$ coarse geometry. The key steps in the proof involve establishing coarse equivalences between inherited metrics and uniform spherical metrics, and applying induction on the dimension. Using this conclusion, we can derive the following theorem.

\begin{theorem}(see Theorem \ref{th 4.8})
Let $W$ be a uniformly discrete bounded geometry metric space, and let $\mathcal{U}_{*}$ be an anti-\v{C}ech sequence for $W$. Let $X=X(W,\mathcal{U}_{*})$ be the corresponding total coarsening space, and let $X_{f}=X(W,\mathcal{U}_{*})_{f}$ be the fusion coarse structure on $X$. Then the $L^{p}$ coarse Baum-Connes conjecture holds for $W$ if and only if $K_{*}(B^{p}(X_{f}))=0$.
\end{theorem}

From this theorem, we see that the obstruction group for the $L^{p}$ coarse Baum-Connes conjecture on the bounded coarse structure is the $K$-theory of the $L^{p}$ Roe algebra of the total coarsening space equipped with a fusion coarse structure. Moreover, we show that the obstructions vanish under the assumption of finite asymptotic dimension, and thus give a new proof of the following theorem. 

\begin{theorem}(see Theorem \ref{th 5.3})
Let $W$ be a bounded geometry metric space with finite asymptotic dimension. For $p\in[1,\infty)$, the $L^{p}$ coarse Baum-Connes assembly map $\mu: KX^{p}_{*}(W)\rightarrow K_{*}(B^{p}(W))$ is an isomorphism.
\end{theorem}

A key step in the proof is to use a more computable hybrid coarse structure. For this structure, we first calculate that the $K$-theory of the $L^{p}$ Roe algebra of the total coarsening space on bounded geometry metric spaces with finite asymptotic dimension is zero, and then prove that the $K$-theory groups of $L^{p}$ Roe algebras of the hybrid and fusion coarse structures are isomorphic. It follows that the obstructions vanish in such spaces.

This paper is organized as follows. In section 2, we will briefly review the basic ingredients for $C_{0}$ coarse geometry and $L^{p}$ coarse $K$-homology. In section 3, we prove the $C_{0}$ version of the $L^{p}$ coarse Baum-Connes conjecture. In section 4, we show that the $L^{p}$ coarse Baum-Connes conjecture holds for a uniformly discrete bounded geometry metric space $W$ if and only if the obstruction group $K_{*}(B^{p}(X_{f}))$ vanishes. Finally, in section 5, we demonstrate that the obstructions vanish under the assumption of finite asymptotic dimension, thereby providing a new proof of the $L^{p}$ coarse Baum-Connes conjecture.
\section{Preliminaries}
In this section, we will briefly recall some basic ingredients of $C_{0}$ coarse geometry and $L^{p}$ coarse $K$-homology. Moreover, we will introduce the notions of $L^{p}$ $C_{0}$ coarse $K$-homology, $L^{p}$ $K$-homology at infinity and $L^{p}$ coarse Baum-Connes assembly map. From the perspective of coarse geometry, we will be interested in proper metric spaces. Say that a metric space is proper if every closed bounded set is compact. Throughout this paper, we assume that $p\in [1,\infty)$.

\subsection{$C_{0}$ coarse geometry and nerve complex}
In this subsection, we review the definitions of $C_{0}$ coarse structure, anti-\v{C}ech sequence and nerve complex. We first recall from Roe's monograph \cite{Roe03} that a coarse structure is a family of subsets that records the large-scale properties of the underlying space.
\begin{definition}\cite{Roe03}
Let $X$ be a set, and let $\mathcal{E}$ be a family of subsets of $X\times X$. We say that $\mathcal{E}$ is a coarse structure on $X$ if it satisfies the following axioms:
\begin{itemize}
\item The diagonal $\bigtriangleup=\{(x,x)\mid x\in X\}$ is a member of $\mathcal{E}$;
\item $E\in\mathcal{E}$ and $F\subseteq E$ implies $F\in\mathcal{E}$;
\item $E,F\in\mathcal{E}$ implies $E\cup F\in \mathcal{E}$;
\item $E\in\mathcal{E}$ implies $E^{-1}:=\{(y,x)\mid (x,y)\in E\}\in\mathcal{E}$;
\item $E, F\in\mathcal{E}$ implies $E\circ F:=\{(x,y)\mid \exists z, s.t. (x,z)\in E \text{ and } (z,y)\in F\}\in\mathcal{E}$;
\end{itemize}
The pair $(X,\mathcal{E})$ is called a coarse space, and the elements of $\mathcal{E}$ are called controlled sets or entourages.
\end{definition}

In coarse spaces, we have the following concept of bounded sets. 
\begin{definition}\cite{Roe03}
Let $X$ be a coarse space, and let $B$ be a subset of $X$. We say that $B$ is bounded if $B\times B$ is controlled.
\end{definition}

The basic definitions of coarse map, closeness and coarse equivalence in coarse geometry are as follows.
\begin{definition}\cite{Roe03}
Let $f: X \rightarrow Y$ be any map between two proper coarse spaces. We say $f$ is 
\begin{itemize}
\item The map $f$ is proper if for any bounded subset $B$ of $Y$, the inverse image $f^{-1}(B)$ is bounded in $X$;
\item The map $f$ is bornologous if for each controlled subset $E$ of $X\times X$, the set $(f\times f)(E)$ is a controlled subset of $Y\times Y$;
\item The map $f$ is coarse if it is proper and bornologous;
\item Two maps $f,g: X\rightarrow Y$ are close if the set $\{(f(x),g(x))\mid x\in X\}$ is a controlled subset of $Y\times Y$;
\item A coarse map $f: X\rightarrow Y$ is called a coarse equivalence if there exists a coarse map $g: Y\rightarrow X$ such that $f\circ g$ and $g\circ f$ are close to the identities on $Y$ and $X$, respectively;
\item The spaces $X$ and $Y$ are coarsely equivalent if there exists a coarse equivalence $f: X\rightarrow Y$.
\end{itemize}
\end{definition}

The following definition plays an important role in studying the large-scale features of the underlying space.
\begin{definition}\cite{Roe96}
Let $(X, d)$ be a proper metric space. The bounded coarse structure on $X$ is the collection $\mathcal{E}$ of all those subsets $E\subseteq X\times X$ for which
$$
\sup\{d(x,y)\mid \text{ for all } (x,y)\in E\}<\infty.
$$
\end{definition}

\begin{example}
With the bounded coarse structure, the integer lattice $\mathbb{Z} ^{n}$ is coarsely equivalent to the $n$-dimensional Euclidean space.
\end{example}
 
In $C_{0}$ coarse geometry, the most fundamental concept is as follows.
\begin{definition}\cite{Wri03}
Let $(X, d)$ be a proper metric space. The $C_{0}$ coarse structure on $X$ is the collection $\mathcal{E}$ of all those subsets $E\subseteq X\times X$ for which for any $\varepsilon>0$, there exists a compact subset $K$ of $X$ such that
$$
 d(x,y)<\varepsilon, \text{ for all }(x,y)\in E\backslash (K\times K).
$$
\end{definition}

\begin{remark}
There is a difference between the $C_{0}$ coarse structure and the bounded coarse structure. In bounded coarse geometry, $S^{2}\times\mathbb{R}$ with the cusp metric is coarsely equivalent to $S^{2}\times\mathbb{R}$ with the product metric, while in $C_{0}$ coarse geometry, they are not coarsely equivalent since the former space is coarsely equivalent to a line, and the latter space is not \cite{Wri03}. Therefore, we denote by $X$ a metric space with the bounded coarse structure, and by $X_{0}$ the same space with the $C_{0}$ coarse structure. 
\end{remark}

Recall that the ordinary \v{C}ech cohomology focuses on finer and finer covers, while the anti-\v{C}ech system proposed by Roe focuses on coarser and coarser covers. To describe this system, we need to review the notions of the Lebesgue number and the diameter of a cover.
Let $\mathcal{U}$ be an open cover of a metric space $X$. Recall that the Lebesgue number of $\mathcal{U}$, denoted $Lebesgue(\mathcal{U})$, is the largest number $R$ such that for every open subset $V$ of $X$ of diameter at most $R$, there exists $U\in\mathcal{U}$ with $V\subseteq U$. The diameter of $\mathcal{U}$, denoted $\Diam(\mathcal{U})$, is the supremum of the diameters of the sets $U$ in $\mathcal{U}$.

\begin{definition}\cite{Roe93}
Let $X$ be a proper metric space, and let $\{\mathcal{U}_{i}\}^{\infty}_{i=1}$ be a sequence of open covers of $X$. We say that $\{\mathcal{U}_{i}\}^{\infty}_{i=1}$ is an anti-\v{C}ech sequence if $$\lim\limits_{i\rightarrow\infty}Lebesgue(\mathcal{U}_{i})=\infty \text{ and } \Diam(\mathcal{U}_{i})\leq Lebesgue(\mathcal{U}_{i+1})<\infty.$$
\end{definition}

In topology, the nerve complex of an open cover is an abstract complex that records the pattern of intersections between open sets in the cover. This notion is defined as follows.
\begin{definition}\cite{Wri05}
Let $X$ be a metric space, and let $\mathcal{U}$ be an open cover of $X$. The nerve $N_{\mathcal{U}}$ of $\mathcal{U}$  is an abstract simplicial complex that has the members of $\mathcal{U}$ as vertices, and where $[U_{1}], [U_{2}], \cdots, [U_{k}]$ span a simplex if and only if $U_{1}\cap\cdots\cap U_{k}\neq\varnothing$.
\end{definition}
\begin{remark}
The anti-\v{C}ech property indicates that for every $V\in\mathcal{U}_{i}$, there exists $U\in\mathcal{U}_{i+1}$ such that $V\subseteq U$. For this reason, there are simplicial connecting maps $$\phi_{i}: N_{\mathcal{U}_{i}}\rightarrow N_{\mathcal{U}_{i+1}}, [V]\mapsto [U].$$
\end{remark}
Here are some concrete examples.
\begin{example}
Let $X=[0,1]$ and $\mathcal{U}=\{U_{1},U_{2}\}$, where $U_{1}=[0,{2}/{3})$ and $U_{2}=({1}/{2},1]$. Then the nerve $N_{\mathcal{U}}$  is an abstract $1$-simplex.
\end{example}

\begin{example}
Let $X=S^{1}$ and $\mathcal{U}=\{U_{1},U_{2},U_{3}\}$, where each $U_{i}$ is an arc covering one third of $S^{1}$, with some overlap with the adjacent $U_{i}$. Then the nerve $N_{\mathcal{U}}$ is an unfilled triangle.
\end{example}

\subsection{$L^{p}$ $C_{0}$ coarse $K$-homology and $L^{p}$ $K$-homology at infinity}
In this subsection, we will introduce the definitions of $L^{p}$ $C_{0}$ coarse $K$-homology and $L^{p}$ $K$-homology at infinity. 

To define two $L^{p}$ $K$-homology theories, we first need to review the notion of an $L^{p}$ module.
\begin{definition}\cite{ZZ21}
Let $X$ be a proper metric space. An $L^{p}$-$X$-module is an $L^{p}$ space $E^{p}_{X}=\ell^{p}(Z_{X})\otimes\ell^{p}$ equipped with a natural pointwise multiplication action of $C_{0}(X)$ by restricting to $Z_{X}$, where $Z_{X}$ is a countable dense subset in $X$, $\ell^{p}=\ell^{p}(\mathbb{N})$, and $C_{0}(X)$ is the algebra of all complex-valued continuous functions on $X$ which vanish at infinity.
\end{definition}

\begin{remark}
For a fixed $p\in [1,\infty)$, different $L^{p}$ modules are isometrically isomorphic, and their corresponding $L^{p}$ Roe algebras and $L^{p}$ localization algebras are also isomorphic. Therefore, we only need to select a suitable $L^{p}$ module in the appropriate situation.
\end{remark}

The following key definition relates the operators on $E^{p}_{X}$ to the structure of $X$.
\begin{definition}\cite{ZZ21}
Let $E^{p}_{X}$ be an $L^{p}$-$X$-module, and let $E^{p}_{Y}$ be an $L^{p}$-$Y$-module. Let $T: E^{p}_{X}\rightarrow E^{p}_{Y}$ be a bounded linear operator. The support of $T$, denoted $\supp(T)$, consists of all points $(x,y)\in X\times Y$ such that $\chi_{V}T\chi_{U}\neq 0$ for all open neighborhoods $U$ of $x$ and $V$ of $y$, where $\chi_{U}$ and $\chi_{V}$ are the characteristic functions of $U$ and $V$, respectively.
\end{definition}

To describe the notions of $L^{p}$ Roe algebra and $L^{p}$ localization algebra, we need to review the following definition.
\begin{definition}\cite{ZZ21}
Let $E^{p}_{X}$ be an $L^{p}$-$X$-module, and let $T$ be a bounded linear operator acting on $E^{p}_{X}$.
\begin{itemize}
\item The propagation of $T$ is defined to be $\prop(T):=\sup\{d(x,y)\mid (x,y)\in\supp(T)\}$;
\item $T$ is said to be locally compact if $\chi_{K}T$ and $T\chi_{K}$ are compact operators for any compact subset $K$ of $X$.
\end{itemize}
\end{definition}

Recall that an operator $T$ on $E^{p}_{X}$ is controlled if the set $\supp(T)$ is controlled. Then we can introduce the notion of $L^{p}$ Roe algebra.

\begin{definition}\label{Roe}
Let $E^{p}_{X}$ be an $L^{p}$-$X$-module. The $L^{p}$ Roe algebra $B^{p}(E^{p}_{X})$ of $X$ is the norm closure of the algebra of all locally compact, controlled operators acting on $E^{p}_{X}$.
\end{definition}

\begin{remark}
The above definition is equivalent to the one proposed by Zhang and Zhou in \cite{ZZ21}. Below, we will write $B^{p}(X)$ for $B^{p}(E^{p}_{X})$, since the algebra $B^{p}(E^{p}_{X})$ does not depend on the choice of the $L^{p}$-$X$-module $E^{p}_{X}$ (see Corollary 2.9 in \cite{ZZ21}). 
\end{remark}

The definition of covering isometries for coarse maps between two $L^{p}$ modules is as follows.
\begin{definition}
Let $E^{p}_{X}$ be an $L^{p}$-$X$-module, and let $E^{p}_{Y}$ be an $L^{p}$-$Y$-module. Let $f: X\rightarrow Y$ be a coarse map. We say that an isometry $V_{f}: E^{p}_{X}\rightarrow E^{p}_{Y}$ is a covering isometry for $f$ if the set $\{(y,f(x))\mid (y,x)\in\supp(V_{f})\}$ is controlled.
\end{definition}

The $L^{p}$ analogue of an involution $V^{*}_{f}$ of an isometry $V_{f}$ in the $C^{*}$-algebra case is  then as follows.
\begin{definition}
Let $f: X\rightarrow Y$ be a coarse map, and let $V_{f}$ be a covering isometry for $f$. We say that a contractible operator $V^{+}_{f}: E^{p}_{Y}\rightarrow E^{p}_{X}$ is an involution of $V_{f}$ if $V^{+}_{f}V_{f}=I$ and $\supp(V^{+}_{f})=\supp(V_{f})^{-1}$.
We call $(V_{f},V^{+}_{f})$ a covering isometry pair for $f$.
\end{definition}

The next key definition is the basic ingredient of $L^{p}$ $K$-homology theories.
\begin{definition}\cite{ZZ21}
Let $X$ be a proper metric space. The $L^{p}$ localization algebra of $X$, denoted by $B^{p}_{L}(X)$, is defined to be the norm closure of the algebra of all bounded and uniformly norm-continuous functions $f$ from $[0,\infty)$ to $B^{p}(X)$ such that
$$
\prop(f(t)) \text{ is uniformly bounded and } \prop(f(t))\rightarrow 0 \text{ as }t\rightarrow\infty.
$$
The propagation of $f$ is defined to be $\max\{\prop(f(t)): t\in[0,\infty)\}$.
\end{definition}

Subsequently, the definition of $L^{p}$ coarse $K$-homology is coming.
\begin{definition}[$L^{p}$ coarse $K$-homology]\cite{ZZ21}\label{def 2.21}
Let $X$ be a proper metric space. The $L^{p}$ coarse $K$-homology groups of $X$ are the groups 
$$KX^{p}_{*}(X):=\varinjlim\limits_{i}K_{*}(B^{p}_{L}(N_{\mathcal{U}_{i}})),$$ where $\mathcal{U}_{i}$ is any anti-\v{C}ech sequence of $X$, and the maps on $K$-homology groups are induced by the connecting maps $\phi_{i}: N_{\mathcal{U}_{i}}\rightarrow N_{\mathcal{U}_{i+1}}$.
\end{definition}

\begin{remark}
In Definition \ref{def 2.21}, the anti-\v{C}ech system can select a fixed sequence since the direct limit does not depend on the choice of anti-\v{C}ech sequences.
\end{remark}

\begin{remark}
The coarsening of a partial order on the collection of all covers of $X$ is defined as follows: 
	$$\mathcal{U}_{1}\preccurlyeq \mathcal{U}_{2} \text{ if for every } U_{1}\in\mathcal{U}_{1}, \text{ there exists } U_{2}\in \mathcal{U}_{2}\text{ such that } U_{1}\subseteq U_{2}.$$
\end{remark}

To define the $L^{p}$ $C_{0}$ coarse $K$-homology, we need to review the notions of locally finite $C_{0}$ open covers and separable spaces. A cover $\mathcal{U}$ of $X$ is called locally finite in $X$ if each point $x\in X$ has a neighborhood ${U}$ in $X$ which intersects only finitely many elements of $\mathcal{U}$. Say that a cover $\mathcal{U}$ of $X$ is a $C_{0}$ open cover if there exists a positive $C_{0}$ function $p$ on $X$ such that for every $U\in \mathcal{U}$, 
$\Diam(U)\leq \inf\{p(x)\mid x\in U\}.$ 
Say that a topological space is separable if it contains a countable dense subset.
\begin{definition}[$L^{p}$ $C_{0}$ coarse $K$-homology]
Let $X$ be a proper separable coarse space, and let $\mathcal{C}_{0}(X)$ denote the collection of all locally finite $C_{0}$ open covers of $X$, directed by the above relation $\mathcal{U}_{1}\preccurlyeq\mathcal{U}_{2}$. The $L^{p}$ $C_{0}$ coarse $K$-homology groups of $X$ are the groups $$KX^{p}_{*}(X_{0}):=\varinjlim\limits_{\mathcal{U}\in \mathcal{C}_{0}(X)}K_{*}(B^{p}_{L}(N_{\mathcal{U}})),$$ where the maps on $K$-homology groups are induced by the connecting maps $\phi: N_{\mathcal{U}_{1}}\rightarrow N_{\mathcal{U}_{2}}$ for $\mathcal{U}_{1}\preccurlyeq \mathcal{U}_{2}$.
\end{definition}

\begin{remark}
Unlike the bounded coarse structure, the inductive limit of $L^{p}$ coarse $K$-homology on the $C_{0}$ coarse structure could take over an uncountable directed system, since the $C_{0}$ structure may not be generated by a countable family of its controlled sets.
\end{remark}

In the next section, we will show that in many spaces, the $L^{p}$ $C_{0}$ coarse $K$-homology is isomorphic to the $L^{p}$ $K$-homology at infinity, which is defined as follows.
\begin{definition}[$L^{p}$ $K$-homology at infinity]
Let $X$ be a locally compact topological space. The $L^{p}$ $K$-homology of $X$ at infinity is 
$$K^{p,\infty}_{*}(X):=\varinjlim\limits_{C\subseteq X \text{ compact}}K_{*}(B^{p}_{L}(X/C)),$$ where the directed system is given by inclusions, and for $C_{1}\subseteq C_{2}$ the map $K_{*}(B^{p}_{L}(X/C_{1}))$$\rightarrow K_{*}(B^{p}_{L}(X/C_{2}))$ is induced by the quotient map.
\end{definition}

\begin{remark}
The term `infinity' here is because when the compact set $C$ is larger and larger, the underlying quotient space $X/C$ only remains the information at infinity.
\end{remark}

\subsection{$L^{p}$ coarse Baum-Connes assembly map}
In this subsection, we will establish the $L^{p}$ coarse Baum-Connes assembly map.

We begin with some definitions of a class of good metrics given in \cite{Wri05}, including the path metric and the uniform spherical metric.
\begin{definition}\cite{Wri05}
Let $(X,d)$ be a proper metric space. The associated path metric is
$$
d_{l}(x,x')=\inf\{l(\gamma)\mid \gamma: [0,1]\rightarrow X \text{ a path with }\gamma(0)=x, \gamma(1)=x'  \}, 
$$
where the length of $\gamma$ is given by
$$
l(\gamma)=\sup\left\{\sum\limits_{i=1}^{n}d(\gamma(t_{i-1}),\gamma(t_{i}))\middle\vert\ n\in\mathbb{N}, 0=t_{0}\leq\cdots\leq t_{n}=1\right\}.
$$
If $d_{l}=d$, we say that $(X,d)$ is a path metric space.
\end{definition}

\begin{definition}\cite{Wri05}
The spherical $m$-simplex is the intersection of $m$-sphere in $\mathbb{R}^{m+1}$ with the positive cone, equipped with the spherical path metric. Barycentric coordinates on this simplex are defined by taking convex combinations of the vertices $(1,0,\cdots,0)$, $\cdots$, $(0,\cdots,0,1)$, and then projecting radially onto the sphere. 
\end{definition}

Say that a simplicial complex is locally finite if each of its vertices belongs to only finitely many simplices. Below, the uniform spherical metric is an optimization of the spherical path metric. 
\begin{definition}\cite{Wri05}\label{def A.4}
A uniform spherical metric on a locally finite simplicial complex is a metric with the following properties:
\begin{itemize}
\item each simplex is isometric to the spherical $m$-simplex;
\item the restriction of the metric to each component is the path metric;
\item the components are far apart, meaning that for $R>0$, there exists a finite subcomplex $K$ such that if $x,y$ lie in different components and $x\notin K$ or $y\notin K$, then $d(x,y)>R$.
\end{itemize}
\end{definition}

To describe the next theorem, we need to review the definition of the uniformly bounded open cover.
\begin{definition}\cite{Wri05}
Let $X$ be a proper separable coarse space, and let $\mathcal{U}$ be an open cover of $X$. We say that $\mathcal{U}$ is a uniformly bounded open cover if $\bigcup\limits_{U\in\mathcal{U}}U\times U$ is controlled.
\end{definition}

A proper separable coarse space is paracompact, thus it admits a locally finite open cover. The following lemma demonstrates this fact.
\begin{lemma}\cite{Wri05}\label{lemma 2.9}
Let $X$ be a proper separable coarse space. Then for any uniformly bounded open cover $\mathcal{U}_{1}$ of $X$, there exists a locally finite uniformly bounded open cover $\mathcal{U}_{2}$ with $\mathcal{U}_{1}\preccurlyeq \mathcal{U}_{2}$.
\end{lemma}

The following technical proposition is a vital step in the next theorem. 
\begin{proposition}\cite{Wri05}\label{prop 2.29}
Let $X$ be a proper separable coarse space, and let $\mathcal{U}$ be a locally finite uniformly bounded open cover of $X$. Equip $N_{\mathcal{U}}$ with a uniform spherical metric, and the corresponding bounded coarse structure. Let $\eta=\eta_{\mathcal{U}}: N_{\mathcal{U}}\rightarrow X$ be any map such that if $y$ lies in the star about a vertex $[V]$ of $N_{\mathcal{U}}$ then $\eta(y)\in V$. Then $\eta$ is coarse and any two such maps are close.
\end{proposition}

\begin{remark}
We have a well-defined map $\varinjlim\limits_{\mathcal{U}}K_{*}(B^{p}(N_{\mathcal{U}}))\rightarrow K_{*}(B^{p}(X))$ as the fact that if $\phi: N_{\mathcal{U}_{1}}\rightarrow N_{\mathcal{U}_{2}}$ is a connecting map, $\eta_{\mathcal{U}_{1}}:N_{\mathcal{U}_{1}}\rightarrow X$ and  $\eta_{\mathcal{U}_{2}}:N_{\mathcal{U}_{2}}\rightarrow X$ are coarse maps obtained by Proposition \ref{prop 2.29},  then $\eta_{\mathcal{U}_{1}}$ is close to $\eta_{\mathcal{U}_{2}}\circ\phi$.
\end{remark}

We need the following identification to formulate the right-hand side of the $L^{p}$ coarse Baum-Connes assembly map. 
\begin{theorem}
Let $X$ be a proper separable coarse space. Then the map 
$$
\varinjlim\limits_{\mathcal{U}}K_{*}(B^{p}(N_{\mathcal{U}}))\rightarrow K_{*}(B^{p}(X))
$$
is an isomorphism, where the direct limit is taken over the directed system of locally finite uniformly bounded open covers $\mathcal{U}$.
\end{theorem}

\begin{proof}
Construct an inverse map of $\eta_{\mathcal{U}}$ obtained by the previous proposition as follows:
$$\psi_{\mathcal{U}}:X\rightarrow N_{\mathcal{U}}$$ 
$$\qquad\qquad\qquad\quad x\mapsto [U] \text{ with }x\in U.$$
Then $\psi_{\mathcal{U}}\circ\eta_{\mathcal{U}}$ and $\eta_{\mathcal{U}}\circ\psi_{\mathcal{U}}$ are close to the identities on $N_{\mathcal{U}}$ and $X$ respectively. However, $\psi_{\mathcal{U}}$ may not be a coarse map. 

First, we show that there are two covering isometries for the identities on $X$ and $N_{\mathcal{U}}$. Let $E^{p}_{X}$ be a represented $L^{p}$-$X$-module of $B^{p}(X)$, and let $E^{p}_{N_{\mathcal{U}}}$ be a represented $L^{p}$-$N_{\mathcal{U}}$-module of $B^{p}(N_{\mathcal{U}})$. We write $$B^{p}(X)=A=\varinjlim\limits_{C}A_{C},$$ where $C$ is any open controlled subset of $X\times X$ containing the diagonal $\bigtriangleup$, and $A_{C}=\{T\in A\mid \supp(T)\subseteq C\}.$
Since the controlled sets $C$ generate the coarse structure, it follows that for any controlled operator $T\in A$, its support $\supp(T)$ must be contained in some $C$, which implies that $T\in A_{C}$ for some $C$ appearing in the direct limit.

Let $\widetilde{\mathcal{U}_{C}}=\{U\mid U\times U\subseteq C\}$, meaning that $\widetilde{\mathcal{U}_{C}}$ is a uniformly bounded open cover. By Lemma \ref{lemma 2.9}, there exists a locally finite uniformly bounded cover $\mathcal{U}_{C}$ with $\widetilde{\mathcal{U}_{C}}\preccurlyeq\mathcal{U}_{C}$. Put $\psi=\psi_{\mathcal{U}_{C}}: X\rightarrow N_{\mathcal{U}_{C}}$ and $\eta=\eta_{\mathcal{U}_{C}}: N_{\mathcal{U}_{C}}\rightarrow X$.
Let $V_{1}$ be a covering isometry for $\eta$. Analogously, we let $V_{2}: E^{p}_{X}\rightarrow E^{p}_{N_{\mathcal{U}}}$ be an isometry for $\psi$ such that the set $\{(y,\psi(x))\mid (y,x)\in\supp(V_{2})\}$ is controlled. Then there exists a contractible operator $V_{2}^{+}: E^{p}_{N_{\mathcal{U}}}\rightarrow E^{p}_{X}$ such that $V^{+}_{2}V_{2}=I$ and $\supp(V^{+}_{2})=\supp(V_{2})^{-1}$. Therefore, $V_{1}V_{2}$ and $V_{2}V_{1}$ are covering isometries for $\eta\circ\psi$ and $\psi\circ\eta$ respectively. Since $\eta\circ\psi$ and $\psi\circ\eta$ are close to the identities, $V_{1}V_{2}$ and $V_{2}V_{1}$ are covering isometries for the identities on $X$ and $N_{\mathcal{U}}$  respectively.  

Next, we show that $Ad_{V_{2}}$ maps $A_{C}$ into $B^{p}(N_{\mathcal{U}_{C}})$. It is easy to see that for any $T\in A$ with $\supp(T)\subseteq C$, $T'=Ad_{V_2}(T)=V_{2}TV^{+}_{2}$ is locally compact. It thus suffices to show that $T'$ has finite propagation. Observe that
$$
\supp(T')\subseteq\{(y,y')\mid \exists (x,x')\in C, s.t. (y,x), (y',x')\in\supp(V_{2})\}.
$$
Therefore, there exists $R>0$ such that if $(y,y')\in\supp(T')$, then there exists $(x,x')\in C$ with 
$d(y,\psi(x)), d(y',\psi(x'))<R$. Since $\widetilde{\mathcal{U}_{C}}\preccurlyeq\mathcal{U}_{C}$ and $(x,x')\in C$, there exists $U\in\mathcal{U}_{C}$ with $x,x'\in U$ such that
$$
d(\psi(x),[U]), d(\psi(x'),[U])\leq\frac{\pi}{2}.
$$
So for all $(y,y')\in\supp(T')$, $$d(y,y')\leq d(y,\psi(x))+d(\psi(x),[U])+d(\psi(x'),[U])+d(y',\psi(x'))<2R+\pi.$$  Thus $Ad_{V_{2}}: A_{C}\rightarrow B^{p}(N_{\mathcal{U}_{C}})$. Besides, we have $Ad_{V_{1}}: B^{p}(N_{\mathcal{U}_{C}})\rightarrow B^{p}(X)=A$ by Proposition \ref{prop 2.29}.

Furthermore, we show that there is an isomorphism between the $K$-theory groups of the subalgebra of $A$ and $B^{p}(N_{\mathcal{U}_{C}})$. Let $B_{C}$ be the subalgebra of $A$ which is mapped into $B^{p}(N_{\mathcal{U}_{C}})$ by $Ad_{V_{2}}$. Then $A_{C}\subseteq B_{C}$, and thus 
$$
A=\varinjlim\limits_{C}A_{C}\subseteq\varinjlim\limits_{C}B_{C}\subseteq A.
$$
Therefore, $\varinjlim\limits_{C}B_{C}=A$. Since $V_{2}V_{1}$ is a covering isometry for the identity on $N_{\mathcal{U}_{C}}$, $Ad_{V_{1}}(B^{p}(N_{\mathcal{U}_{C}}))\subseteq B_{C}$. Hence, $$K_{*}(B^{p}(N_{\mathcal{U}_{C}}))\rightarrow K_{*}(B_{C})\rightarrow K_{*}(B^{p}(N_{\mathcal{U}_{C}}))$$
is the identity. Since $V_{1}V_{2}$ is a covering isometry for the identity on $X$, it follows that $K_{*}(B^{p}(N_{\mathcal{U}_{C}}))\cong K_{*}(B_{C})$.

Finally, for a controlled open set $C'$ with $C\subseteq C'$, we get a commutative diagram:
$$
\begin{CD}
K_{*}(B^{p}(N_{\mathcal{U}_{C}}))@>{\cong}>> K_{*}(B_{C})\\
@VVV@VVV\\
K_{*}(B^{p}(N_{\mathcal{U}_{C'}}))@>{\cong}>> K_{*}(B_{C'}).
\end{CD}
$$
Taking the limit over $C$, we obtain that $$\varinjlim\limits_{C}K_{*}(B^{p}(N_{\mathcal{U}_{C}}))\cong\varinjlim\limits_{C}K_{*}(B_{C})\cong  K_{*}(B^{p}(X)).$$
\end{proof}

By the isomorphism of $\varinjlim\limits_{\mathcal{U}}K_{*}(B^{p}(N_{\mathcal{U}}))$ and  $K_{*}(B^{p}(X))$, we obtain the following $L^{p}$ coarse Baum-Connes assembly map.
\begin{definition}[$L^{p}$ coarse Baum-Connes assembly map]
Let $X$ be a proper separable coarse space. The $L^{p}$ coarse Baum-Connes assembly map for $X$ is the evaluation-at-zero homomorphism
$$
\mu: KX^{p}_{*}(X)\cong\varinjlim\limits_{\mathcal{U}}K_{*}(B^{p}_{L}(N_{\mathcal{U}}))\rightarrow\varinjlim\limits_{\mathcal{U}}K_{*}(B^{p}(N_{\mathcal{U}}))\cong K_{*}(B^{p}(X)).
$$
\end{definition}

\begin{remark}
Different coarse structures of the metric space $X$ correspond to different directed systems. In general,  when $X$ is equipped with the bounded coarse structure, the direct limit is taken over an anti-\v{C}ech sequence of $X$. However, if $X$ is equipped with the $C_{0}$ coarse structure, the direct limit could take over an uncountable system $\mathcal{C}_{0}(X)$ of $C_{0}$ open covers. In this case, we call the above map the $C_{0}$ version of the $L^{p}$ coarse Baum-Connes assembly map. 
\end{remark}

\section{The $C_{0}$ version of the $L^{p}$ coarse Baum-Connes conjecture}
In this section, we will establish the $L^{p}$ version of the Eilenberg swindles and the six-term Mayer-Vietoris sequence. In addition, we show that the $L^{p}$ coarse Baum-Connes conjecture for $C_{0}$ coarse geometry holds for any infinite uniformly discrete proper metric space. Hence, we establish the $C_{0}$ version of the $L^{p}$ coarse Baum-Connes conjecture for any finite-dimensional simplicial complex. 

\subsection{Mayer-Vietoris sequence and Eilenberg swindles}
In this subsection, we will set up the two fundamental techniques for calculating the $K$-theory groups of  $L^{p}$ Roe algebras, namely, the $L^{p}$ version of the Eilenberg swindles and the six-term Mayer-Vietoris sequence.

To discuss the Mayer-Vietoris sequence, we need the notion of coarsely excisive cover for a proper metric space.
\begin{definition}\cite{HRY93}
Let $X=E\cup F$ be a cover of a proper coarse space $X$ by closed subsets. We say that the cover is coarsely excisive if for any controlled set $A$, there exists a controlled set $B$ such that
$$
E_{A}\cap F_{A}\subseteq (E\cap F)_{B},
$$ 
where $E_{A}=\{x\in X\mid \exists y\in E \text{ with }(x,y)\in A\text{ or }x=y\}$.
\end{definition}

\begin{example}\cite{HRY93}
Let $X=\mathbb{R}$, $E=\{x\in\mathbb{R}\mid x\geq 0\}$ and $F=\{x\in\mathbb{R}\mid x\leq 0\}$. Then the decomposition $X=E\cup F$ is coarsely excisive.
\end{example}

\begin{remark}\cite{Wri05}
On a path metric space $X$, any cover of $X$ by closed subsets is coarsely excisive for the $C_{0}$ and bounded coarse structures.
\end{remark}

The next key theorem shows that for the coarsely excisive cover of a proper coarse space, there is a six-term Mayer-Vietoris sequence in the $K$-theory of the $L^{p}$ Roe algebra.
\begin{theorem}\label{th 3.2}
For a coarsely excisive decomposition $X=Y\cup Z$, there exists a six-term Mayer-Vietoris sequence:
$$
\begin{CD}
K_{1}(B^{p}(Y\cap Z))@>{i_{1*}\oplus i_{2*}}>>K_{1}(B^{p}(Y))\oplus K_{1}(B^{p}(Z))@>{j_{1*}-j_{2*}}>> K_{1}(B^{p}(X))
\\@A{\partial}AA@.@VV{\partial}V\\
K_{0}(B^{p}(X))@<{j_{1*}-j_{2*}}<<K_{0}(B^{p}(Y))\oplus K_{0}(B^{p}(Z))@<{i_{1*}\oplus i_{2*}}<< K_{0}(B^{p}(Y\cap Z)),
\end{CD}
$$
where $i_{1}: Y\cap Z\rightarrow Y$, $i_{2}: Y\cap Z\rightarrow Z$, $j_{1}: Y\rightarrow X$, $j_{2}: Z\rightarrow X$.
\end{theorem}

\begin{proof}
This theorem can be proved in a similar way as shown in \cite{HRY93}. Unlike the case of $C^{*}$-algebras, a homomorphism between $L^{p}$ operator algebras may not have a closed range. Therefore, we need to use another approach to prove that $B^{p}(Y)+B^{p}(Z)$ is closed in $B^{p}(X)$. Let $\{T^{(1)}_{n}+T^{(2)}_{n}\}$ be a Cauchy sequence in $B^{p}(Y)+B^{p}(Z)$, and put $T_{n}\triangleq T^{(1)}_{n}+T^{(2)}_{n}$ in $B^{p}(X)$, then there exists an operator $T\in B^{p}(X)$ such that $T_{n}\xrightarrow{\Vert\cdot\Vert}T$. As $T^{(1)}_{n}$ is the truncation of $T_{n}$ to $Y\times Y$, thus $\Vert T^{(1)}_{n}\Vert\leq\Vert T_{n}\Vert$, we conclude that $\{T^{(1)}_{n}\}$ is a Cauchy sequence. Since $T^{(2)}_{n}=T_{n}-T^{(1)}_{n}$, we can deduce that $\{T^{(2)}_{n}\}$ is also a Cauchy sequence. Hence, there exist $T^{(1)}\in B^{p}(Y)$ and $T^{(2)}\in B^{p}(Z)$ such that $T^{(1)}_{n}\xrightarrow{\Vert\cdot\Vert}T^{(1)}$ and $T^{(2)}_{n}\xrightarrow{\Vert\cdot\Vert}T^{(2)}$, which implies  $T=T^{(1)}+T^{(2)}\in B^{p}(Y)+B^{p}(Z)$. It is easy to verify that $B^{p}(Y)+B^{p}(Z)$ is dense in $B^{p}(X)$. It follows that $B^{p}(Y)+B^{p}(Z)=B^{p}(X)$ as required.
\end{proof}

\begin{remark}
Analogously to Theorem \ref{th 3.2}, we have six-term Mayer-Vietoris sequences in the $L^{p}$ coarse $K$-homology and the $L^{p}$ $K$-homology at infinity.
\end{remark}

The following definition provides a sufficient condition for constructing Eilenberg swindles.
\begin{definition}\cite{Wri05}
Let $X$ be a proper separable coarse space, and let $\{\alpha_{k}\}$ be a sequence of coarse maps on $X$. The sequence $\{\alpha_{k}\}$
\begin{itemize}
\item is properly supported if for every bounded set $B$, the intersection $B\cap Range(\alpha_{k})\neq\emptyset$ for only finitely many $k$;
\item is uniformly controlled if for every controlled set $A$, there exists a controlled set $B_{A}$ such that $(x,x')\in A$ implies $(\alpha_{k}(x),\alpha_{k}(x'))\in B_{A}$ for all $k$;
\item has uniformly close steps if there exists a controlled set $C$ such that $(\alpha_{k}(x),\alpha_{k+1}(x))$ belongs to $C$ for all $k$ and all $x\in X$.
\end{itemize}
\end{definition}

We now turn to construct the Eilenberg swindles, i.e. `infinite sums' of a given element of a $K$-group.

\begin{lemma}[Eilenberg swindles]\label{lemma 3.11}
Let $X$ be a proper separable coarse space, and let $\{\alpha_{k}\}$ be a sequence of coarse maps on $X$. If $\alpha_{0}$ is the identity and the sequence $\{\alpha_{k}\}$ is properly supported, uniformly controlled, and has uniformly close steps, then $K_{*}(B^{p}(X))=0$.
\end{lemma}

\begin{proof}
Let $E^{p}_{X}$ be an $L^{p}$ representation space of $B^{p}(X)$. We claim that there exists a sequence $\{V_{k}\}^{\infty}_{k=0}$ of covering isometries for $\{\alpha_{k}\}^{\infty}_{k=0}$ such that for every element $[T]\in K_{*}(B^{p}(X))$ with $T$ an idempotent or invertible in $M_{n}(B^{p}(X))$ for some $n\in\mathbb{N}$, then $[\beta_{2}(T)]=[\beta_{3}(T)]$ on $K$-theory, where 
$$
\beta_{2}(T)=\begin{pmatrix}T&&\\&Ad_{V_{1}}(T)&\\&&\ddots\end{pmatrix}, \quad \beta_{3}(T)=\begin{pmatrix}Ad_{V_{1}}(T)&&\\&Ad_{V_{2}}(T)&\\&&\ddots\end{pmatrix},
$$
are well-defined operators on $\bigoplus\limits_{k}E^{p}_{X}$, here $Ad_{V_{k}}(T)=V_{k}TV^{+}_{k}$. Set
$$
\beta_{1}(T)=\begin{pmatrix}T&&\\&0&\\&&\ddots\end{pmatrix}.
$$
Then $[\beta_{1}(T)]\bigoplus [\beta_{3}(T)]=[\beta_{2}(T)]$, which implies $[\beta_{1}(T)]=0$ on $K$-theory. Since the top left corner inclusion $i: E^{p}_{X}\rightarrow \bigoplus\limits_{k}E^{p}_{X}$ is a covering isometry for the indentity on $X$, it follows that $[T]=0$ on $K_{*}(B^{p}(X))$. As $[T]$ is arbitrary, we have $K_{*}(B^{p}(X))=0$.

In the following, we prove this claim in three steps. First, we will show that $\bigoplus\limits_{k=0}^{\infty}Ad_{V_{k}}(T)$ is a controlled operator if $T$ is a controlled operator. Let $\mathcal{O}\subseteq X\times X$ be a controlled open neighborhood of the diagonal $\bigtriangleup$, and let $\{V_{k}\}^{\infty}_{k=0}$ be a covering isometry for $\{\alpha_{k}\}^{\infty}_{k=0}$ and 
$$
\supp(V_{k})\subseteq\{(x,x')\in X\times X\mid (x,\alpha_{k}(x'))\in\mathcal{O}\}.
$$
Then, we choose $V_{0}$ as the identity on $E^{p}_{X}$ since $\alpha_{0}$ is the identity on $X$. Observe that if $T$ is a controlled operator in $B^{p}(X)$, then $A=\supp(T)$ is a controlled set. Because $\{\alpha_{k}\}$ is uniformly controlled, there exists a controlled set $B_{A}$ such that $(x,x')\in A$ implies $(\alpha_{k}(x),\alpha_{k}(x'))\in B_{A}$. Then 
$$\supp(Ad_{V_{k}}(T))=\supp(V_{k}TV^{+}_{k})\subseteq\supp(V_{k})\circ\supp(T)\circ\supp(V^{+}_{k})$$
$$\qquad\qquad\quad\qquad\subseteq\mathcal{O}\circ B_{A}\circ\mathcal{O}^{-1}.$$ 
Therefore, $Ad_{V_{k}}(T)$ is a controlled operator, and so is $\bigoplus\limits_{k=0}^{\infty}Ad_{V_{k}}(T)$. 

The next step is to demonstrate that $\bigoplus\limits_{k=0}^{\infty}Ad_{V_{k}}(T)$ belongs to $B^{p}(X)$, for all $T\in B^{p}(X)$. For a bounded set $B$, let $B'=B_{\bigtriangleup^{-1}}$ be the $\bigtriangleup^{-1}$-neighborhood of $B$, so $B'$ is bounded. Note that if $B'\cap Range(\alpha_{k})=\varnothing$, then $\supp(Ad_{V_{k}}(T))\cap (B\times B)=\varnothing.$ Since $\{\alpha_{k}\}$ is properly supported, in other words, $B'\cap Range(\alpha_{k})\neq\varnothing$ for  finitely many $k$, which implies $\supp(Ad_{V_{k}}(T))\cap (B\times B)\neq\varnothing$ for finitely many $k$. Clearly, if $T$ is locally compact, then $Ad_{V_{k}}(T)$ is locally compact, thus $\bigoplus\limits_{k=0}^{\infty}Ad_{V_{k}}(T)$ is locally compact. Consequently, if $T$ is a locally compact controlled operator in $B^{p}(X)$ for the representation on $E^{p}_{X}$, then $\bigoplus\limits_{k=0}^{\infty}Ad_{V_{k}}(T)$ is a locally compact controlled operator in $B^{p}(X)$ for the representation on $\bigoplus\limits_{k} E^{p}_{X}$. Furthermore, by Definition \ref{Roe}, we see that for every $T\in B^{p}(X)$, there exists a sequence $\{T_{\lambda}\}$ of locally compact and controlled operators in $B^{p}(X)$ such that $T_{\lambda}\xrightarrow{\Vert\cdot\Vert}T$,  which implies that $\bigoplus\limits_{k=0}^{\infty}Ad_{V_{k}}(T_{\lambda})\xrightarrow{\Vert\cdot\Vert}\bigoplus\limits_{k=0}^{\infty}Ad_{V_{k}}(T)$. So we have $\bigoplus\limits_{k=0}^{\infty}Ad_{V_{k}}(T)\in B^{p}(X)$, for all $T\in B^{p}(X)$ as required. 

The final step is to prove $[\beta_{2}(T)]=[\beta_{3}(T)]$ on $K$-theory. Set 
$$
U_{k}=\begin{pmatrix}V_{k+1}V^{+}_{k}&1-V_{k+1}V^{+}_{k+1}\\1-V_{k}V^{+}_{k}&V_{k}V^{+}_{k+1}\end{pmatrix},\quad U^{+}_{k}=\begin{pmatrix}V_{k}V^{+}_{k+1}&1-V_{k}V^{+}_{k}\\1-V_{k+1}V^{+}_{k+1}&V_{k+1}V^{+}_{k}\end{pmatrix}.
$$
It is easy to verify that $U_{k}U^{+}_{k}=U^{+}_{k}U_{k}=I$, namely $U^{-1}_{k}=U^{+}_{k}$, and 
$$
U_{k}\begin{pmatrix}V_{k}&0\\0&0\end{pmatrix}=\begin{pmatrix}V_{k+1}&0\\0&0\end{pmatrix},\quad \begin{pmatrix}V^{+}_{k}&0\\0&0\end{pmatrix}U^{+}_{k}=\begin{pmatrix}V^{+}_{k+1}&0\\0&0\end{pmatrix}.
$$
Let 
$$
U=\begin{pmatrix}U_{0}&&\\&U_{1}&\\&&\ddots\end{pmatrix}.\text{ Then }U^{-1}=\begin{pmatrix}U^{+}_{0}&&\\&U^{+}_{1}&\\&&\ddots\end{pmatrix},
$$
and $$U\begin{pmatrix}\beta_{2}(T)&0\\0&0\end{pmatrix}U^{-1}=\begin{pmatrix}\beta_{3}(T)&0\\0&0\end{pmatrix}.$$
Hence, $[\beta_{2}(T)]=[\beta_{3}(T)]$ on $K$-theory, completing the claim.
\end{proof}
\subsection{The case of finite-dimensional simplicial complexes}
The next few results in this subsection will mainly be used to get the $C_{0}$ version of the $L^{p}$ coarse Baum-Connes conjecture for finite-dimensional simplicial complexes. We achieve this goal in two steps. First, we show that the $C_{0}$ version of the conjecture holds for the $0$-dimensional simplicial complex, namely a discrete metric space. Next, using the method of graph theory, we only need a finite number of cutting-and-pasting steps to reduce a finite-dimensional simplicial complex to the 0-dimensional case. Hence, we can show that the $L^{p}$ coarse Baum-Connes assembly map is an isomorphism on the $C_{0}$ structure. 

Below, we will consider uniformly discrete metric spaces, since every metric space has a uniformly discrete subspace that is coarsely equivalent to it. Recall that a metric space $X$ is uniformly discrete if there exists $\varepsilon>0$ such that $d(x_{1},x_{2})\geq\varepsilon$ for all $x_{1}\neq x_{2}$ of $X$. We now proceed to prove the $C_{0}$ version of the $L^{p}$ coarse Baum-Connes conjecture for any infinite uniformly discrete metric space.

\begin{theorem}\label{th 3.17}
Let $X$ be an infinite uniformly discrete proper metric space. Then the $L^{p}$ coarse Baum-Connes conjecture holds for $X$ equipped with the $C_{0}$ coarse structure. Moreover, any two such spaces are coarsely equivalent and
$$
K^{p,\infty}_{*}(X)\cong KX^{p}_{*}(X_{0})\cong K_{*}(B^{p}(X_{0})).
$$
\end{theorem}

\begin{proof}
First, we show that the $L^{p}$ $C_{0}$ coarse $K$-homology is isomorphic to the $L^{p}$ $K$-homology at infinity. For a $C_{0}$ controlled operator $T$ on $X$, its support $\supp(T)$ contains only finitely many points off the diagonal $\bigtriangleup$. Then a $C_{0}$ open cover $\mathcal{U}$ is a cover in which all but finitely many sets are singletons. Set $C=\bigcup\limits_{i=1}^{n}C_{i}$, where $C_{i}$ is a non-singleton and $n$ is some positive integer. Thus the nerve $N_{\mathcal{U}}$ is the union of a simplicial star associated with $C$ and the discrete space $X\backslash C$. It is straightforward to check that the canonical map $f: N_{\mathcal{U}}\rightarrow X/ C$ is a homotopy equivalence. Thanks to the homotopy invariance of the $L^{p}$ coarse $K$-homology, we have $K_{*}(B^{p}_{L}(N_{\mathcal{U}}))\cong K_{*}(B^{p}_{L}(X/C))$. Taking the limit on both sides, we have $$\varinjlim\limits_{\mathcal{U}\in \mathcal{C}_{0}(X)}K_{*}(B_{L}^{p}(N_{\mathcal{U}}))\cong\varinjlim\limits_{C\subseteq X,\text{ compact }}K_{*}(B^{p}_{L}(X/C)).$$
Therefore $$KX^{p}_{*}(X_{0})\cong K^{p,\infty}_{*}(X).$$

For the second part, we explicitly compute the above group in the even case. We calculate that $K_{0}(B^{p}_{L}(X/C))=\prod\limits_{x\in X/C}\mathbb{Z}$ for any compact set $C$, and for $C\subseteq C'$, the map $$s_{1}:\prod_{x\in X/C}\mathbb{Z}=K_{0}(B^{p}_{L}(X/C))\rightarrow K_{0}(B^{p}_{L}(X/C'))=\prod\limits_{x\in X/C'}\mathbb{Z}$$
$$\qquad\qquad (n_{x})_{x\in X/C}\qquad\qquad\mapsto\qquad \left(\sum\limits_{x\in C'/C}n_{x},(n_{x})_{x\in X/C'}\right)$$
is surjective, and $\mathrm{Ker}s_{1}=\{(n_{x})_{x\in X/C}\mid \sum\limits_{x\in C'/C}n_{x}=0, n_{x}=0 \text{ on } X/C' \}$. Thus the homomorphism $$s_{2}:K_{0}(B^{p}_{L}(X))\rightarrow \varinjlim\limits_{C\subseteq X,\text{ compact }}K_{0}(B^{p}_{L}(X/C))$$ is surjective, and $$\mathrm{Ker}s_{2}=\{(n_{x})_{x\in X}\mid\sum\limits_{x\in C}n_{x}=0, n_{x}\text{ is finitely supported }\}.$$

To see the $L^{p}$ assembly map in this space, we represent $C_{0}(X)$ on an $L^{p}$-$X$-module $E^{p}_{X}=\ell^{p}(Z_{X})\otimes\ell^{p}$, where $Z_{X}$ is a countable dense subset in $X$. Then a locally compact operator $T$ with $\supp(T)\subseteq \bigtriangleup$ can be written as $T=(T_{x})_{x\in X}$ for $T_{x}$ on $\ell^{p}$. The $L^{p}$ assembly map is as follows:
$$
\prod\limits_{x\in X}\mathbb{Z}=K_{0}(B^{p}_{L}(X))\rightarrow K_{0}(B^{p}(X_{0}))
$$
$$
\qquad(n_{x}\geq 0)_{x\in X}\mapsto (P_{x})_{x\in X},
$$
where $P_{x}$ is a diagonally supported idempotent with $rank(P_{x})=n_{x}$.

For each locally compact controlled operator $T$ for $X$, we can write $T=T_{C}+ T'_{C}$, where $C\subseteq X$ is a finite subset, $T_{C}$ is a compact operator with $\supp(T_{C})\subseteq C\times C$ and $T'_{C}$ is a diagonal operator over $X\backslash C$ with compact entries. Therefore, for the $C_{0}$ coarse structure, we may write the $L^{p}$ Roe algebra $B^{p}(X_{0})=\varinjlim\limits_{C\subseteq X,\text{ finite }}A_{C}$, where $A_{C}=\mathcal{K}(\ell^{p}(C)\otimes\ell^{p})\oplus\ell^{\infty}(X\backslash C)\otimes\mathcal{K}(\ell^{p}).$ Then we have the map corresponding to the above $L^{p}$ assembly map

$$\prod\limits_{x\in X}\mathbb{Z}=K_{0}(B^{p}_{L}(X))\rightarrow K_{0}(A_{C})=\mathbb{Z}\oplus\prod\limits_{x\in X\backslash C}\mathbb{Z}$$
 $$\qquad\qquad(n_{x})_{x\in X}\mapsto \left(\sum\limits_{x\in C}n_{x},(n_{x})_{x\in X\backslash C}\right).$$
Taking the limit, we deduce that the map $$s_{3}: K_{0}(B^{p}_{L}(X))\rightarrow \varinjlim\limits_{C\subseteq X,\text{ finite }}K_{0}(A_{C})=K_{0}(B^{p}(X_{0}))$$ is surjective, and 
$$\mathrm{Ker}s_{3}=\{(n_{x})_{x\in X}\mid\sum_{x\in C}n_{x}=0, n_{x} \text{ is properly supported }\}.$$
Therefore, $\mathrm{Ker}s_{3}=\mathrm{Ker}s_{2}$. Finally, the first isomorphism theorem gives that $$K_{0}(B^{p}_{L}(X))/\mathrm{Ker}s_2\cong KX_{0}^{p}(X_{0})\text{ and }K_{0}(B^{p}_{L}(X))/\mathrm{Ker}s_{3}\cong K_{0}(B^{p}(X_{0})),$$
which implies an isomorphism $KX_{0}^{p}(X_{0})\cong K_{0}(B^{p}(X_{0}))$. In the odd case, $K_{1}(B^{p}_{L}(X/C))=0$, taking the limit over $C$, we obtain $K^{p,\infty}_{1}(X)=0$, and thus $KX_{1}^{p}(X_{0})=0$. Also, $K_{1}(A_{C})=0$ implies $K_{1}(B^{p}(X_{0}))=0$. Hence, $KX_{1}^{p}(X_{0})\cong K_{1}(B^{p}(X_{0}))$. In conclusion, the $L^{p}$ assembly map is an isomorphism.
\end{proof}

To prove the main theorem of this subsection, we need several definitions and lemmas. Specifically, we will use the concept of a finite binary rooted tree from graph theory to construct a nice binary decomposition of a metric space to obtain the desired isomorphisms.

Recall from \cite{BM08} that a tree is a connected acyclic graph that contains no cycles, and any two vertices are connected by exactly one path in a tree. Say that a rooted tree $T(x)$ is a tree $T$ with a specified vertex $x$, called the root of $T$. Then the definition of the finite binary rooted tree is as follows.

\begin{definition}\cite{Wri05}
For a finite binary rooted tree (each vertex having zero or two successors), a vertex is a leaf if it has no successors, and a fork if it has two successors.
\end{definition}

From the previous definition, we have the notion of binary decomposition of a metric space. 
\begin{definition}\cite{Wri05}
Let $X$ be a metric space, and let $\mathcal{T}$ be a finite binary rooted tree. We say that $\mathcal{T}$ is a binary decomposition of $X$ if $\mathcal{T}$ has a labeling of the vertices by subsets of $X$ for which 
\begin{itemize}
\item the root is labeled by $X$;
\item if $v$ is a fork having successors labeled by $Y$ and $Z$, then $v$ is labeled by $Y\cup Z$. 
\end{itemize}
\end{definition}
The binary decomposition category, denoted by $\mathcal{L(T)}$, whose objects are the labels of $\mathcal{T}$ and the intersection $Y\cap Z$ with $Y$, $Z$ the labels of successors of some fork of $\mathcal{T}$, and whose morphisms are inclusions. 

The concept of admissible decomposition will be useful, as it is a sufficient condition for the next technical lemma.
\begin{definition}\cite{Wri05}
Let $\mathcal{T}$ be a binary decomposition of $X$, and let $h^{1}_{*}$, $h^{2}_{*}$ be covariant functors from $\mathcal{L(T)}$ to the category of $\mathbb{Z}_{2}$-graded abelian groups. Let $\alpha$ be a natural transformation from $h^{1}_{*}$ to $h^{2}_{*}$. We say that the decomposition $\mathcal{T}$ is admissible for $h^{1}$, $h^{2}$ and $\alpha$, if the following conditions are satisfied:
\begin{itemize}
\item if $v$ is a fork having successors labeled by $Y$, $Z$, then the decomposition $Y\cup Z$ is excisive for $h^{1}$ and $h^{2}$, in other words, there are Mayer-Vietoris exact sequences
$$\cdots\rightarrow h^{i}_{*}(Y\cap Z)\rightarrow h^{i}_{*}(Y)\oplus h^{i}_{*}(Z)\rightarrow h^{i}_{*}(Y\cup Z)\rightarrow\cdots$$
and naturality of $\alpha$ extends to the boundary maps;
\item if $v$ is a fork having successors labeled by $Y$, $Z$, then $\alpha: h^{1}_{*}(Y\cap Z)\rightarrow h^{2}_{*}(Y\cap Z)$ is an isomorphism;
\item if $v$ is a leaf labeled by $Y$, then $\alpha: h^{1}_{*}(Y)\rightarrow h^{2}_{*}(Y)$ is an isomorphism.
\end{itemize}
\end{definition}

\begin{example}
Let $X$ be a $3$-simplex equipped with a uniform spherical metric, and let $\alpha: KX_{*}^{p}(\cdot)\rightarrow K_{*}(B^{p}(\cdot))$. As shown in the proof of Theorem \ref{th 3.18}, we can construct an admissible binary decomposition of $X$. See Figure \ref{Fig: 1} for more details.
\end{example}

The core idea behind the following lemma is to use a finite number of cutting-and-pasting techniques by a Mayer-Vietoris long exact sequence to obtain the isomorphism. 
\begin{lemma}\cite{Wri05}\label{lemma 3.16}
If there exists a binary decomposition $\mathcal{T}$ of $X$ which is admissible for functors $h^{1}_{*}$, $h^{2}_{*}$ and natural transformation $\alpha$, then $\alpha: h^{1}_{*}(X)\rightarrow h^{2}_{*}(X)$ is an isomorphism.
\end{lemma}

The next important lemma ties the inherited metric to the uniform spherical metric. To describe it, let us recall the notions of the distortion and the relatively connected subspace. The distortion depicts the difference between the inherited metric and the path metric of the subspace.
\begin{definition}\cite{Wri05}
Let $(X,d)$ be a path metric space, and let $Y$ be a closed subset of $X$. Let $d_{l}$ be the path metric on $Y$ associated with the metric inherited from $X$. The distortion of $Y$ in $X$ is
$$
\sup\left\{\frac{d_{l}(y,y')}{d(y,y')}\middle\vert\ y,y'\in Y, y\neq y'\right\}.
$$
The distortion of a path $\gamma:[0,1]\rightarrow Y$ is 
$$
\sup\left\{\frac{l(\gamma|_{[t,t']})}{d(\gamma(t),\gamma(t'))} \middle\vert\ 0\leq t<t'\leq 1\right\}.
$$
\end{definition}

For simplicity, let us focus on the case of the relatively connected subspace.
\begin{definition}\cite{Wri05}
A subspace $Y$ of a topological space $X$ is relatively connected if each connected component of $X$ contains at most one connected component of $Y$.
\end{definition}

There are two metrics on the $n$-th barycentric subdivision $X^{(n)}$ of a finite-dimensional simplicial complex $X$. The one is the inherited metric, the other is the uniform spherical metric. The following lemma indicates that these two metrics are bi-Lipschitz equivalent, and the Lipschitz constant depends only on $n$ and the dimension of $X$. Hence, the two metrics are coarsely equivalent on the $C_{0}$ and bounded structure.
\begin{lemma}\label{lemma A.8}\cite{Wri05}
Let $X$ be a finite-dimensional simplicial complex equipped with the uniform spherical metric. Let $Y$ be a subcomplex of $X^{(n)}$, the $n$-th barycentric subdivision of $X$, and let $Y_{\sigma}=Y\cap\sigma$ with $\sigma$ a simplex of $X$. Assume that
\begin{itemize}
\item for each simplex $\sigma$ of $X$, the subcomplex $Y_{\sigma}$ of $X^{(n)}$ is connected;
\item each edge of $X$ has non-empty intersection with $Y$.
\end{itemize}
Then there is a uniform bound on the distortion of each component of $Y$ in $X$, depending only on $n$ and the dimension of $X$. Moreover, suppose $Y$ is relatively connected in $X$. In that case, the inherited metric on $Y$ is coarsely equivalent and $C_{0}$-coarsely equivalent to any uniform spherical metric on $Y$, and this coarse equivalent is bi-Lipschitz on components.
\end{lemma}

As a consequence, we obtain the $L^{p}$ coarse Baum-Connes conjecture for $C_{0}$ coarse geometry in the finite-dimensional case.
\begin{theorem}\label{th 3.18}
Let $X$ be a finite-dimensional simplicial complex equipped with a uniform spherical metric, and let $X_{0}$ be the space $X$ with the $C_{0}$ coarse structure. Let $p\in [1,\infty)$, then 
$$
K^{p,\infty}_{*}(X)\cong KX^{p}_{*}(X_{0})\underset{\cong}{\xrightarrow{\mu}}K_{*}(B^{p}(X_{0})).
$$
\end{theorem}
\begin{proof}
The idea is to construct a binary decomposition of $X$ and show that this is admissible for the following two transformations
$$
K^{p,\infty}_{*}(\cdot)\rightarrow KX^{p}_{*}(\cdot)\text{ and }KX^{p}_{*}(\cdot)\rightarrow K_{*}(B^{p}(\cdot)).
$$
Let $X^{(2)}$ denote the second barycentric subdivision of $X$, and let $Y_{k}$ be the union of simplicial stars in $X^{(2)}$ about the barycentres of the $k$-simplices of $X$. Then $X=Y_{0}\cup Y_{1}\cup Y_{2}\cup\cdots\cup Y_{m}$, where $\dim X=m$. See Figure \ref{Fig: 1}. To see these stars are uniformly separated, it suffices to consider pairs of stars $(star(x)$, $star(y))$ in the same component of $X$. Since the metric on $X$ is a path metric,  the distance between two stars is $d(star(x), star(y))=\ell(\gamma)$, the length of the shortest path $\gamma$  between the boundaries of the stars. If the stars are in a common simplex $\sigma$ of $X$, then the distance between stars is the length of the path within $\sigma$. Otherwise, there exists a simplex $\sigma$ with $x\in\sigma$ such that the path between $star(x)$ and $star(y)$ must connect the boundary of the $star(x)$ to a face of $\sigma$ not containing $x$. In either case, we get a lower bound on the distance between the stars, which is independent of the simplex $\sigma$ because all simplices of the same dimension are isometric by the first item of Definition \ref{def A.4}.
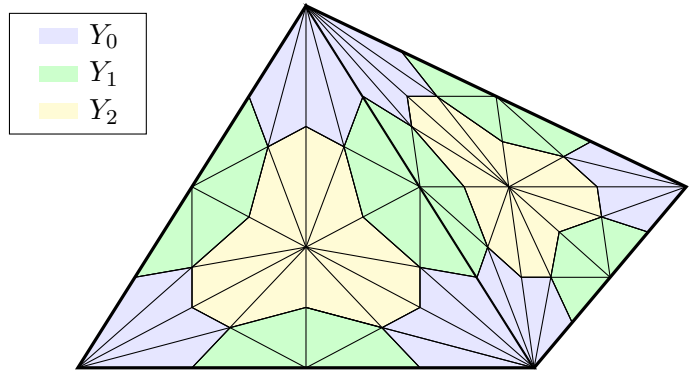
\begin{figure}[h]
\begin{tikzpicture}
\draw[fill=yellow!20] (2,0.53)--(3,0.8)--(4,0.53)--(4.5,0.8)--(4.5,1.33)--(3.75,2)--(3.5,2.93)--(3,3.2)--(2.5,2.93)--(2.25,2)--(1.5,1.33)--(1.5,0.8)--cycle;
\draw[fill=green!20](3,0.8)--(2,0.53)--(1.5,0)--(4.5,0)--(4,0.53)--cycle;
\draw[fill=green!20](2.5,2.93)--(2.25,2)--(1.5,1.33)--(0.75,1.2)--(2.25,3.6)--cycle;
\draw[fill=green!20](3.5,2.93)--(3.75,2)--(4.5,1.33)--(5.25,1.2)--(3.75,3.6)--cycle;
\draw[fill=blue!10](4.5,0)--(6,0)--(5.25,1.2)--(4.5,1.33)--(4.5,0.8)--(4,0.53)--cycle;
\draw[fill=blue!10](1.5,1.33)--(1.5,0.8)--(2,0.53)--(1.5,0)--(0,0)--(0.75,1.2)--cycle;
\draw[fill=blue!10](3.5,2.93)--(3,3.2)--(2.5,2.93)--(2.25,3.6)--(3,4.8)--(3.75,3.6)--cycle;
\draw (3,0)--(4.5,0.8)--(4.5,2.4)--(3,3.2)--(1.5,2.4)--(1.5,0.8)--cycle;
\draw (3,1.6)--(1.5,0);
\draw (3,1.6)--(4.5,0);
\draw (3,1.6)--(5.25,1.2);
\draw (3,1.6)--(3.75,3.6);
\draw (3,1.6)--(2.25,3.6);
\draw (3,1.6)--(0.75,1.2);
\draw (6,0)--(1.5,2.4);
\draw (3,0.8)--(6,0);
\draw (2.25,2)--(0,0)--(3,0.8)--(6,0)--(3.75,2)--(3,4.8)--cycle;
\draw (3,4.8)--(3,0);
\draw[very thick] (0,0)--(3,4.8)--(6,0)--cycle;
\draw[fill=yellow!20] (5.83,1.2)--(6.22,1.2)--(6.33,1.8)--(6.89,2)--(6.83,2.4)--(6.39,2.8)--(5.58,3)--(4.72,3.6)--(4.33,3.6)--(4.39,3.2)--(5.08,2.4)--(5.39,1.6)--cycle;
\draw[fill=green!20](6.33,1.8)--(6.89,2)--(7.5,1.8)--(6.5,0.6)--(6.22,1.2)--cycle;
\draw[fill=green!20](5.58,3)--(6.39,2.8)--(6.75,3)--(4.25,4.2)--(4.72,3.6)--cycle;
\draw[fill=green!20](5.08,2.4)--(4.39,3.2)--(3.75,3.6)--(5.25,1.2)--(5.39,1.6)--cycle;
\draw[fill=blue!10](6,0)--(6.5,0.6)--(6.22,1.2)--(5.83,1.2)--(5.39,1.6)--(5.25,1.2)--cycle;
\draw[fill=blue!10](8,2.4)--(6.75,3)--(6.39,2.8)--(6.83,2.4)--(6.89,2)--(7.5,1.8)--cycle;
\draw[fill=blue!10](3,4.8)--(4.25,4.2)--(4.72,3.6)--(4.33,3.6)--(4.39,3.2)--(3.75,3.6)--cycle;
\draw (0,0)--(4.5,2.4);
\draw (6,0)--(5.5,3.6);
\draw (8,2.4)--(4.5,2.4);
\draw (3,4.8)--(7,1.2);
\draw (5.67,2.4)--(6.5,0.6);
\draw (5.67,2.4)--(7.5,1.8);
\draw (5.67,2.4)--(6.75,3);
\draw (5.67,2.4)--(4.25,4.2);
\draw (5.67,2.4)--(3.75,3.6);
\draw (5.67,2.4)--(5.25,1.2);
\draw (4.5,2.4)--(5.83,1.2)--(7,1.2)--(6.83,2.4)--(6.39,2.8)--(5.5,3.6)--(4.33,3.6)--cycle;
\draw (3,4.8)--(5.08,2.4)--(6,0)--(6.33,1.8)--(8,2.4)--(5.58,3)--cycle;
\draw[very thick] (3,4.8)--(8,2.4)--(6,0);
\draw (-0.9,3.1) rectangle (0.9,4.7);
\path[fill=blue!10] (-0.5,4.3) rectangle (0,4.5);
\node at (0,4.4)[right]{\small$Y_{0}$};
\path[fill=green!20] (-0.5,3.8) rectangle (0,4);
\node at (0,3.9)[right]{\small$Y_{1}$};
\path[fill=yellow!20] (-0.5,3.3) rectangle (0,3.5);
\node at (0,3.4)[right]{\small$Y_{2}$};
\end{tikzpicture}
\caption{The second barycentric subdivision of a simplex, decomposed by $\{Y_{k}\}$.}
\label{Fig: 1}
\end{figure}

To build a desired binary decomposition, we replace $Y_{k}$ with $\tilde{Y}_{k}$, the union of $Y_{k}$ and the 1-skeleton of $X^{(2)}$, since it is difficult to prove that the binary decomposition $X=(((Y_{0}\cup Y_{1})\cup Y_{2})\cup\cdots)\cup Y_{m}$ is admissible. Also, as $\tilde{Y}_{k}$ is relatively connected, we can apply Lemma \ref{lemma A.8}.  Meanwhile, let $G_{k}$ be the graph consisting of those edges of $X^{(2)}$ which are not contained in $Y_{k}$, let $V_{k}$ be the uniformly separated stars about the vertices of $G_{k}$, and let $E_{k}$ be the uniformly separated segments in the edges of $G_{k}$. 

We construct the following binary decomposition:
\begin{align*}
 X &=Z_{m} (\text{ the root })\\
 Z_{k} &=Z_{k-1}\cup\widetilde{Y}_{k}\\
Z_{0} &=\widetilde{Y}_{0}\\
\widetilde{Y}_{k} &=Y_{k}\cup G_{k}\\
G_{k} &=V_{k}\cup E_{k}.
\end{align*}
Clearly, $Y_{k}$, $V_{k}$ and $E_{k}$ are leaves and $Z_{k}$, $\widetilde{Y}_{k}$ and $G_{k}$ are forks.
Below, we show that this decomposition is admissible. Firstly, we show that each fork is excisive. Observe that each of $Z_{k}$, $\widetilde{Y}_{k}$ and $G_{k}$ is relatively connected in $X$. By Lemma \ref{lemma A.8}, the inherited metrics are $C_{0}$-coarsely equivalent to uniform spherical metrics on leaves. Therefore, for the uniform spherical metrics, the following decompositions are excisive:
\begin{align*}
Z_{k}&=Z_{k-1}\cup \widetilde{Y}_{k}, Z_{k-1}, \widetilde{Y}_{k} \text{ are closed and } Z_{k-1}\cap\widetilde{Y}_{k}\neq\varnothing;\\
\widetilde{Y}_{k}&=Y_{k}\cup G_{k}, Y_{k}, G_{k}\text{ are closed and }Y_{k}\cap G_{k}\neq\varnothing; \\
G_{k}&=V_{k}\cup E_{k}, V_{k}, E_{k}\text{ are closed and }V_{k}\cap E_{k}\neq\varnothing.
\end{align*}
Secondly, we prove that at each fork, the intersection yields isomorphisms for two transformations. We proceed by induction on the dimension $m$ of $X$. In the case $m=0$, the space $X$ is uniformly discrete, so the result follows from Theorem \ref{th 3.17}. It is straightforward to check that the intersections $Y_{k}\cap G_{k}$ and $V_{k}\cap E_{k}$ are 0-dimensional, thus they induce isomorphisms. Further, we observe that $\dim(Z_{k-1}\cap\widetilde{Y}_{k})=m-1$. Besides, $Z_{k-1}\cap\widetilde{Y}_{k}$ is relatively connected since it contains the 1-skeleton of $X^{(2)}$. Then by Lemma \ref{lemma A.8}, its inherited metric is $C_{0}$ coarsely equivalent to the uniform spherical metric.  By inductive hypothesis, we obtain that $Z_{k-1}\cap\widetilde{Y}_{k}$ yields isomorphisms.

Finally, we demonstrate that each leaf $Y_{k}$, $V_{k}$, and $E_{k}$ yields isomorphisms for two transformations. We first consider the case of $Y_{k}$. It is easy to see that each $Y_{k}$ is either compact if it consists of finitely many stars, or is coarsely homotopy equivalent to the infinite uniformly discrete space consisting of the $k$-barycenters. Indeed, the homotopy is continuous and even contractive on each star. Since the stars are uniformly separated, it is a $C_{0}$ coarse homotopy. Hence, by Theorem \ref{th 3.17}, we obtain that the transformations are isomorphisms on each $Y_{k}$. In the same manner, the result is also true for $V_{k}$ and $E_{k}$. Consequently, according to Lemma \ref{lemma 3.16}, we obtain two isomorphisms as required.
\end{proof}

\begin{remark}
The finite dimension condition is necessary. For an infinite-dimensional example [\cite{Wri05}, Remark 3.19], it is easy to check that the $L^{p}$ coarse Baum-Connes conjecture fails. 
\end{remark}

\section{The obstructions for the $L^{p}$ coarse Baum-Connes conjecture}
In this section, we recall the notions of the coarsening space and the fusion coarse structure in \cite{Wri05}. According to the $L^{p}$ coarse Baum-Connes conjecture for $C_{0}$ coarse geometry (Theorem \ref{th 3.18}), we reformulate both sides of the $L^{p}$ coarse Baum-Connes conjecture using the $C_{0}$ and fusion coarse structures. Finally, on the fusion structure, we create an obstruction group to the $L^{p}$ coarse Baum-Connes conjecture.

\subsection{The coarsening space}
To build a `geometric' obstruction group to the $L^{p}$ coarse Baum-Connes conjecture, we need to coarsen the metric space into a simplicial complex. More precisely, we will review the definitions of total coarsening space, partial coarsening space and fusion coarse structure.

Suppose $X$ and $Y$ are topological spaces, $A$ is a closed subspace of $Y$, and $f: A\rightarrow X$ is a continuous map. Recall from \cite{Lee13} that an adjunction space is formed by the quotient space
$$
X\cup_{f}Y=(X\amalg Y)/\sim,
$$
where $\sim$ denotes the equivalence relation on the disjoint union $X\amalg Y$ generated by $a\sim f(a)$ for all $a\in A$.
 
\begin{definition}[The total coarsening space]\cite{Wri05}
Let $W$ be a metric space, and let $\mathcal{U}_{*}$ be an anti-\v{C}ech sequence for $W$. Let $\phi_{i}: N_{\mathcal{U}_{i}}\rightarrow N_{\mathcal{U}_{i+1}}$ be the connecting maps. The coarsening space of $W$ and $\mathcal{U}_{*}$ is an adjunction space
$$
X=X(W,\mathcal{U}_{*})=N_{\mathcal{U}_{1}}\times [1,2]\cup_{\phi_{1}}N_{\mathcal{U}_{2}}\times [2,3]\cup_{\phi_{2}}\cdots,
$$
where each simplex $\sigma$ is given the spherical metric, and $X$ is equipped with the largest path metric bounded above by the product metric on each $\sigma\times [i,i+1]$. Denote by $\pi: X\rightarrow [1,\infty)$ the map on $X$ arising from the projection maps $N_{\mathcal{U}_{i}}\times [i, i+1]\rightarrow [i,i+1]$.
\end{definition}

\begin{definition}[The partial coarsening space]\cite{Wri05}
The partial coarsening spaces of $W$, $\mathcal{U}_{*}$ are the spaces
$$
X_{i}=X_{i}(W,\mathcal{U}_{*})=N_{\mathcal{U}_{1}}\times [1,2]\cup_{\phi_{1}}\cdots\cup_{\phi_{i-1}}N_{\mathcal{U}_{i}}
$$
equipped with the metrics inherited as subspaces of the coarsening space. In other words, $X_{i}=\pi^{-1}([1,i])$.
\end{definition}

\begin{example}
Let $X=S^{1}$ and $\mathcal{U}_{1}=\{V_{1},V_{2},V_{3}\}$, where each $V_{i}$ is an arc covering one third of $S^{1}$, with some overlap with the adjacent $V_{i}$. Then the nerve $N_{\mathcal{U}_{1}}$ is an unfilled triangle. Let $\mathcal{U}_{2}=\{U_{1}, U_{2}\}$ such that $U_{2}=V_{3}$ and $V_{1}$, $V_{2}$ are contained in $U_{1}$, which implies that $\mathcal{U}_{1}\preccurlyeq \mathcal{U}_{2}$, and $N_{\mathcal{U}_{2}}$ is a $1$-simplex. Put $\mathcal{U}_{3}=\mathcal{U}_{2}$, meaning that $N_{\mathcal{U}_{3}}=N_{\mathcal{U}_{2}}$. Then $X_{1}=N_{\mathcal{U}_{1}}$, $X_{2}=N_{\mathcal{U}_{1}}\times [1,2]\cup_{\phi_{1}}N_{\mathcal{U}_{2}}$ and $X_{3}=N_{\mathcal{U}_{1}}\times [1,2]\cup_{\phi_{1}}N_{\mathcal{U}_{2}}\times [2,3]\cup_{\phi_{2}}N_{\mathcal{U}_{3}}$. The partial coarsening space of $X$ is shown in Figure \ref{Fig: 2}.
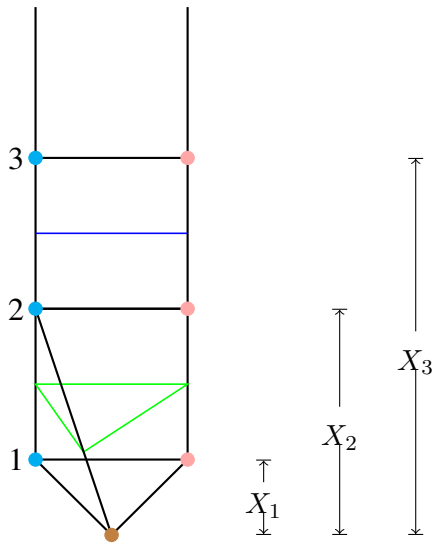
\begin{figure}[h]
\begin{tikzpicture}
\draw[thick] (0,2)--(1,1)--(2,2);
\draw[thick](0,2)--(0,4)--(2,4)--(2,2)--cycle;
\draw[semithick,green!100](0,3)--(2,3)--(0.63,2.1)--cycle;
\draw[thick](1,1)--(0,4);
\draw[thick](0,6)--(2,6)--(2,4)--(0,4)--cycle;
\draw[semithick,blue](0,5)--(2,5);
\draw[thick](0,6)--(0,8);
\draw[thick](2,6)--(2,8);
\draw[fill,brown](1,1)circle(2.5pt);
\draw[fill,cyan](0,2)circle(2.5pt);
\node at (0,2) [left]{1};
\draw[fill,cyan](0,4)circle(2.5pt);
\node at (0,4) [left]{2};
\draw[fill,cyan](0,6)circle(2.5pt);
\node at (0,6) [left]{3};
\draw[fill,red!35](2,2)circle(2.5pt);
\draw[fill,red!35](2,4)circle(2.5pt);
\draw[fill,red!35](2,6)circle(2.5pt);
\draw[|<-](3,1)--(3,1.3);
\node at (3,1.7) [below]{\small$X_{1}$};
\draw[->|](3,1.7)--(3,2);
\draw[|<-](4,1)--(4,2.2);
\node at (4,2.6) [below]{\small$X_{2}$};
\draw[->|](4,2.7)--(4,4);
\draw[|<-](5,1)--(5,3.2);
\node at (5,3.6) [below]{\small$X_{3}$};
\draw[->|](5,3.7)--(5,6);
\end{tikzpicture}
\caption{The partial coarsening space of $S^{1}$.}
\label{Fig: 2}
\end{figure}
\end{example}

The next powerful definition is beneficial in showing that the $L^{p}$ Roe algebra of the total coarsening space has trivial $K$-theory for numerous coarse structures.
\begin{definition}\cite{Wri05}
The collapsing map from $X$ to $\pi^{-1}([t,\infty))$ is the map
$$
\Phi_{t}(x,s)=\begin{cases}
(\phi_{j'-1}\circ\cdots\circ\phi_{j}(x),t), & \text{ for }(x,s)\in N_{\mathcal{U}_{j}}\times [j,j+1),\\ & \text{ with } s\leq t, t\in [j',j'+1),\\
(x,s), & \text{ for } s\geq t.
\end{cases}
$$
In fact, they are all contractive maps, and $\Phi_{t'}\circ\Phi_{t}=\Phi_{t'}$ when $t>t'$. 
\end{definition}

The following lemma tells us that the collapsing map may have an excellent contracting property if a suitable anti-\v{C}ech sequence is chosen.
\begin{lemma}\cite{Wri05}\label{lemma 4.4}
Let $W$ be a countable discrete metric space, and let $\mathcal{U}_{*}$ be an anti-\v{C}ech sequence of covers of $W$. There exist a subsequence $\mathcal{U}_{i_{k}}$ and a sequence of connecting maps $\phi_{i_{k}}: N_{\mathcal{U}_{i_{k}}}\rightarrow N_{\mathcal{U}_{i_{k+1}}}$ such that for $X=X(W, \mathcal{U}_{i_{*}})$ and for each compact subset $K$ of $X$, there exists $t$ such that $\Phi_{t}(K)$ is a point.
\end{lemma}

When discussing the $L^{p}$ coarse Baum-Connes conjecture for both bounded and $C_{0}$ coarse structures, we usually assume that the metric space has bounded geometry. This assumption serves to eliminate potential counterexamples in both the bounded and $C_{0}$ cases. Therefore, we will consider uniformly discrete bounded geometry metric spaces. Recall that a uniformly discrete metric space $X$ has bounded geometry if for all $R>0$, the number of points in the $R$-ball $B_{R}(x)$ is bounded and independent of $x\in X$.
Moreover, Lemma \ref{lemma 3.11} and Lemma \ref{lemma 4.4} yield that the $K$-theory of the $L^{p}$ Roe algebra of the total coarsening space with the $C_{0}$ coarse structure is zero.

\begin{theorem}\label{th 4.5}
Let $W$ be a uniformly discrete bounded geometry metric space, and let $\mathcal{U}_{*}$ be an anti-\v{C}ech sequence for $W$. Then $B^{p}(X(W,\mathcal{U}_{*})_{0})$ has trivial $K$-theory. 
\end{theorem}

\begin{proof}
This theorem can be proved by the same method as employed in [\cite{Wri05}, Theorem 4.5]. To be specific, we construct a sequence of maps $\{\alpha_{k}\}$ satisfying the hypotheses of Lemma \ref{lemma 3.11} as follows:

Put $$r_{k}: \mathbb{R}^{+}\rightarrow\mathbb{R}^{+},\quad r_{k}(t)=\begin{cases}logk-t, & 0\leq t\leq logk\\ 0, & t\geq log k\end{cases}.$$
We choose a basement $x_{0}\in X$ and define$$\alpha_{k}: X\rightarrow X, x\mapsto \Phi_{r_{k}(d(x,x_{0}))}(x).$$
The verification that $\{\alpha_{k}\}$ is properly supported, uniformly controlled, and has uniformly close steps is left to the readers.
\end{proof}

\subsection{Reformulation of both sides of the $L^{p}$ coarse Baum-Connes conjecture}
In this subsection, we will reconstruct the left-hand side of the $L^{p}$ coarse Baum-Connes conjecture using the ideal in the $L^{p}$ Roe algebra of the total coarsening space equipped with the $C_{0}$ coarse structure, and reformulate the right-hand side of the $L^{p}$ coarse Baum-Connes conjecture using a corresponding ideal in the $L^{p}$ Roe algebra on the fusion coarse structure.

To rewrite the left-hand side of the $L^{p}$ coarse Baum-Connes conjecture, we introduce the following ideal.
\begin{definition}
Let $I_{0}=I_{0}(W,\mathcal{U}_{*})=\varinjlim\limits_{i}B^{p}(X_{i}(W,\mathcal{U}_{*})_{0})$.
\end{definition}
One may regard $I_{0}$ as an ideal of $B^{p}(X(W,\mathcal{U}_{*})_{0})$. Indeed, every $L^{p}$ Roe algebra $B^{p}(X_{i}(W,\mathcal{U}_{*})_{0})$ is contained in $B^{p}(X(W,\mathcal{U}_{*})_{0})$, and the closure of their union forms an ideal $I_{0}$.

It is a basic fact that the $L^{p}$ coarse $K$-homology encodes the information about the small-scale topological structure of a metric space for any coarse structure. Combining this fact and the $L^{p}$ coarse Baum-Connes conjecture for $C_{0}$ coarse geometry (Theorem \ref{th 3.18}), we can deduce that the $L^{p}$ coarse $K$-homology on the bounded structure can be identified with the $K$-theory of the ideal $I_{0}$ on the finer $C_{0}$ structure.  

\begin{theorem}\label{th 4.7}
Let $W$ be a uniformly discrete bounded geometry metric space, and let $\mathcal{U}_{*}$ be an anti-\v{C}ech sequence for $W$. Then $KX^{p}_{*}(W)\cong K_{*}(I_{0}(W,\mathcal{U}_{*}))$.
\end{theorem} 

\begin{proof}
First, we show that $KX^{p}_{*}(W)\cong\varinjlim\limits_{i}K_{*}(B^{p}(N_{\mathcal{U}_{i}})_{0})$. Since for any compact subset $C$ of $N_{\mathcal{U}_{i}}$, there exists a sufficiently large integer $j$ such that $\phi_{i+j}\circ\cdots\circ\phi_{i}(C)$ is a vertex of $N_{\mathcal{U}_{i+j+1}}$, we obtain 
$$
KX^{p}_{*}(W):=\varinjlim\limits_{i}K_{*}(B^{p}_{L}(N_{\mathcal{U}_{i}}))\cong\varinjlim\limits_{i}\varinjlim\limits_{C\subseteq N_{\mathcal{U}_{i}},\text{ compact }}K_{*}(B^{p}_{L}(N_{\mathcal{U}_{i}}/C)).
$$
By Theorem \ref{th 3.18}, we have
$$
KX^{p}_{*}(W)\cong\varinjlim\limits_{i}KX^{p}_{*}((N_{\mathcal{U}_{i}})_{0})\cong\varinjlim\limits_{i}K_{*}(B^{p}(N_{\mathcal{U}_{i}})_{0}).
$$
Next, we demonstrate that the inclusion $\pi^{-1}\{i\}\hookrightarrow X_{i}$ gives rise to an isomorphism on $K$-theory. Thanks to Theorem \ref{th 4.5}, we know that $K_{*}(B^{p}(X_{0}))=0$. Taking the anti-\v{C}ech sequence as $\mathcal{U}_{i}$, $\mathcal{U}_{i+1},\cdots$ in Theorem \ref{th 4.5}, we get $K_{*}(B^{p}(\pi^{-1}[i,\infty))_{0})=0$. It is straightforward to check that the decomposition $X=X_{i}\cup\pi^{-1}[i,\infty)$ is coarsely excisive and $X_{i}\cap\pi^{-1}[i,\infty)=\pi^{-1}\{i\}$. According to Theorem \ref{th 3.2}, there exists a six-term Mayer-Vietoris sequence:
$$  \xymatrix@C=0.6cm@R=0.6cm{
  K_{1}(B^{p}(\pi^{-1}\{i\})_{0})\ar[r] & K_{1}(B^{p}(X_{i})_{0})\oplus K_{1}(B^{p}(\pi^{-1}[i,\infty))_{0})\ar[r]&  K_{1}(B^{p}(X_{0}))\ar[d]_{\partial}\\
   K_{0}(B^{p}(X_{0})) \ar[u]_{\partial}& K_{0}(B^{p}(X_{i})_{0})\oplus K_{0}(B^{p}(\pi^{-1}[i,\infty))_{0})\ar[l]& K_{0}(B^{p}(\pi^{-1}\{i\})_{0}).\ar[l]}$$
Furthermore, we yield an isomorphism $K_{*}(B^{p}(\pi^{-1}\{i\})_{0})\cong K_{*}(B^{p}(X_{i})_{0})$ induced by the inclusion $\pi^{-1}\{i\}\hookrightarrow X_{i}$.

Moreover, we prove that $\pi^{-1}\{i\}$ is $C_{0}$ coarsely equivalent to $N_{\mathcal{U}_{i}}$. Note that $N_{\mathcal{U}_{i}}$ is equipped with a uniform spherical metric, and $\pi^{-1}\{i\}$ has the metric inherited from $X$. Let $\varepsilon<2$ and $\varepsilon$ be less than the distance between any two components of $N_{\mathcal{U}_{i}}$. Then $d(x,y)=d_{N_{\mathcal{U}_{i}}}(x,y)$ when $d(x,y)<\varepsilon$. Therefore, these two metrics are $C_{0}$ coarsely equivalent. Thus, we obtain an isomorphism $B^{p}(\pi^{-1}\{i\})_{0}\cong B^{p}(N_{\mathcal{U}_{i}})_{0}$.

As a result, $B^{p}(N_{\mathcal{U}_{i}})_{0}\rightarrow B^{p}(X_{i})_{0}$ induces an isomorphism on $K$-theory. Hence
$$
KX^{p}_{*}(W)=\varinjlim\limits_{i}K_{*}(B^{p}(N_{\mathcal{U}_{i}})_{0})\cong \varinjlim\limits_{i}K_{*}(B^{p}(X_{i})_{0})\cong K_{*}(I_{0}).
$$
\end{proof}

Recall from \cite{Wri05} that the fusion coarse structure is a structure finer than the bounded coarse structure but not as fine as the $C_{0}$ coarse structure.

\begin{definition}[The fusion coarse structure]\cite{Wri05}
Let $X=X(W,\mathcal{U}_{*})$ be the total coarsening space, and let $X_{i}=X_{i}(W,\mathcal{U}_{*})$ be the partial coarsening spaces. The fusion coarse structure on $X$, denoted $X_{f}=X(W,\mathcal{U}_{*})_{f}$, is the coarse structure for which a set $A\subseteq X\times X$ is controlled if and only if
\begin{itemize}
\item $d|_{A}$ is bounded, i.e. $A$ is controlled for the bounded coarse structure, and
\item there exists $i$ such that $d|_{{A\backslash(X_{i}\times X_{i})}}$ is $C_{0}$.
\end{itemize}
\end{definition}

\begin{remark}
The first item of this definition means that for any partial coarsening space $X_{i}$, this structure is indeed a bounded structure.
\end{remark}

To reformulate the right-hand side of the $L^{p}$ coarse Baum-Connes conjecture, we introduce the next ideal.
\begin{definition}
Let $I_{f}=I_{f}(W,\mathcal{U}_{*})=\varinjlim\limits_{i}B^{p}(X_{i}(W,\mathcal{U}_{*}))$.
\end{definition}
One can consider $I_{f}$ as an ideal of $B^{p}(X(W,\mathcal{U}_{*})_{f})$ as the facts that the bounded and fusion coarse structures agree on $X_{i}(W,\mathcal{U}_{*})$, and every $L^{p}$ Roe algebra $B^{p}(X_{i}(W,\mathcal{U}_{*}))$ is included in $B^{p}(X(W,\mathcal{U}_{*})_{f})$, then the closure of their union forms an ideal $I_{f}$.

The next technical lemma illustrates the calculation of the distance from a simplex to a vertex in the same component. 
\begin{lemma}\cite{Wri05}\label{lemma A.5}
Let $X$ be a locally finite simplicial complex with a uniform spherical metric. Then the distance from a vertex $[V]$ of $X$ to a simplex $\sigma$ in the same component of $X$ is ${\pi}/{2}$ times the length (i.e. number of edges) of the shortest simplicial path from {\color{red}$[V]$} to $\sigma$. In other words, there are no shortcuts through the interior of a simplex.
\end{lemma}

As a consequence, we show that the $K$-theory of the $L^{p}$ Roe algebra on the bounded structure can be identified with the $K$-theory of the ideal $I_{f}$ on the fusion structure, and hence obtain an equivalent statement of the $L^{p}$ coarse Baum-Connes conjecture.  
\begin{theorem}\label{th 4.12}
There is an isomorphism $K_{*}(I_{f})\cong K_{*}(B^{p}(W))$, and moreover the forgetful map $I_{0}\hookrightarrow I_{f}$ gives rise to the following commutative diagram:
\begin{equation}\label{eq 1}
\begin{CD}
KX^{p}_{*}(W)@>{\mu}>> K_{*}(B^{p}(W))\\
@V{\cong}VV@VV{\cong}V\\
K_{*}(I_{0})@>>> K_{*}(I_{f}).
\end{CD}
\end{equation}
The $L^{p}$ coarse Baum-Connes conjecture is equivalent to the statement that the forgetful map $I_{0}\hookrightarrow I_{f}$ induces an isomorphism on $K$-theory.
\end{theorem}

\begin{proof}
For each $i$, we construct a coarse map $\psi: W\rightarrow X_{i}$ with the composition $N_{\mathcal{U}_{i}}\rightarrow W_{i}\rightarrow X_{i}$ close to the canonical inclusion $N_{\mathcal{U}_{i}}\hookrightarrow X_{i}$. The diagram (\ref{eq 1}) is the limit of the following diagrams
\begin{equation}\label{eq 2}
 \xymatrix{
  K_{*}(B^{p}_{L}(N_{\mathcal{U}_{i}}))\ar[r]\ar[d] & K_{*}(B^{p}(N_{\mathcal{U}_{i}}))\ar[r]&  K_{*}(B^{p}(W)) \ar[d]\\
  K_{*}(B^{p}(X_{i})_{0}) \ar[rr]& & K_{*}(B^{p}(X_{i})),}
\end{equation}
which commutes for each $i$. Thanks to Theorem \ref{th 4.7}, we see that $KX^{p}_{*}(W)\cong K_{*}(I_{0})$. It thus suffices to prove that $K_{*}(B^{p}(W))\rightarrow K_{*}(B^{p}(X_{i}))$ is an isomorphism for every $i$. The identification will be proved by showing that the coarse maps $\psi: W\rightarrow X_{i}$ are coarse equivalences. 

First, we construct a map $\zeta$ such that $\psi\circ\zeta$ and $\zeta\circ\psi$ are close to the identities. Let $\psi: W\rightarrow\pi^{-1}\{i\}\hookrightarrow X_{i}$ be any map which maps $w\in W$ to a vertex $[V]$ of $\pi^{-1}\{i\}$ with $w\in V$.  Let $\eta: N_{\mathcal{U}_{i}}\rightarrow W$ be any map taking $x$ in the star of a vertex $[V]$ to $\eta(x)$ in $V$. It is easy to verify that $\psi\circ\eta$ is close to the canonical inclusion $N_{\mathcal{U}_{i}}\hookrightarrow X_{i}$. Define
$$
\zeta: X_{i}\xrightarrow{\Phi_{i}}N_{\mathcal{U}_{i}}\xrightarrow{\eta}W, \zeta=\eta\circ\Phi_{i}.
$$
Indeed, for $x\in X_{i}$, $d(\psi\circ\zeta(x),x)=d(\psi\circ\eta\circ\Phi_{i}(x),x)\leq\pi+i$, and for $w\in W$, $d(\zeta\circ\psi(w),w)\leq \Diam(\mathcal{U}_{i})<\infty$.
Thus $\psi\circ\zeta$ and $\zeta\circ\psi$ are close to the identity on $X_{i}$ and $W$ respectively.
 
Next, we show that $\zeta$ is a coarse map. For $x, x'\in X_{i}$ and $d(x,x')<2j$, there exists a path between $\Phi_{i+j}(x)$ and $\Phi_{i+j}(x')$ with length at most $2j$ in $N_{\mathcal{U}_{i+j}}$. According to Lemma \ref{lemma A.5}, there exists a sequence of open sets $V_{0}, V_{1}, \cdots, V_{k}$ in $\mathcal{U}_{i+j}$ with $V_{i}\cap V_{i+1}\neq\varnothing$ for $i=0,\cdots, k-1$, $\zeta(x)\in V_{0}$, $\zeta(x')\in V_{k}$ and $k\leq\frac{4j}{\pi}+2.$  Therefore, when $d(x,x')<2j$, we have 
$$d(\zeta(x),\zeta(x'))\leq d(V_{0},V_{k})\leq (k+1)\Diam(\mathcal{U}_{i+j})\leq(\frac{4j}{\pi}+3)\Diam(\mathcal{U}_{i+j}).$$ Certainly, $\zeta$ is proper, so $\zeta$ is coarse. 

Lastly, we prove that $\psi$ is also a coarse map. For $w,w'\in W$ and $d(w,w')<R$, we put $j\geq 0$ such that $Lebesgue(\mathcal{U}_{i+j})\geq R$. Then there exists a vertex $[V]\in N_{\mathcal{U}_{i+j}}$ with $w,w'\in W$ such that $\Phi_{i+j}(\psi(w))$ and $\Phi_{i+j}(\psi(w'))$ are vertices adjacent to $[V]$ in $N_{\mathcal{U}_{i+j}}$, thus $d(\Phi_{i+j}(\psi(w)), \Phi_{i+j}(\psi(w')))\leq\pi$. Since $\psi(w), \psi(w')$ are vertices of $\pi^{-1}\{i\}$, $d(\psi(w), \Phi_{i+j}(\psi(w)))\leq j$ and $d(\psi(w'), \Phi_{i+j}(\psi(w')))\leq j$, which means $d(\psi(w),\psi(w'))\leq 2j+\pi$. Obviously, $\psi$ is proper, thus it is coarse. In conclusion, $\psi$ and $\zeta$ are coarse equivalences.
Hence, $K_{*}(B^{p}(W))\cong  K_{*}(B^{p}(X_{i}))$ in diagram (\ref{eq 2}), passing to the limit, we obtain the isomorphism $K_{*}(B^{p}(W))\cong  K_{*}(I_{f})$.
\end{proof}

\begin{remark}
According to Theorem \ref{th 4.12}, we can establish the $L^{p}$ coarse Baum-Connes conjecture by showing that the forgetful map $I_{0}\hookrightarrow I_{f}$ induces an isomorphism on $K$-theory. This method presents a significant advantage within the framework of coarse geometry. Indeed, the ideals $I_{0}$ and $I_{f}$ belong to $L^{p}$ Roe algebras of the $C_{0}$ and fusion coarse structures on the total coarsening space $X$. These ideal structures receive robust support from coarse geometric tools, significantly enhancing the effectiveness in proving the conjecture.
\end{remark}
\subsection{The obstruction group for the $L^{p}$ coarse Baum-Connes conjecture}
In this subsection, we conclude from the above discussion that the obstruction group for the $L^{p}$ coarse Baum-Connes conjecture is the $K$-theory of the $L^{p}$ Roe algebra of the total coarsening space equipped with a fusion coarse structure.
\begin{theorem}\label{th 4.8}
Let $W$ be a uniformly discrete bounded geometry metric space, and let $\mathcal{U}_{*}$ be an anti-\v{C}ech sequence for $W$. Let $X=X(W,\mathcal{U}_{*})$ be the corresponding total coarsening space, and let $X_{f}=X(W,\mathcal{U}_{*})_{f}$ be the fusion coarse structure on $X$. Then the $L^{p}$ coarse Baum-Connes conjecture holds for $W$ if and only if $K_{*}(B^{p}(X_{f}))=0$.
\end{theorem}

\begin{proof}
The outline of the proof is as follows. Firstly, we verify that $I_{0}=B^{p}(X_{0})\cap I_{f}$. Clearly, $I_{0}\subseteq B^{p}(X_{0})\cap I_{f}$. On the other hand, for any $T\in B^{p}(X_{0})\cap I_{f}$, $T=\varinjlim\limits_{i}T_{i}$, where $T_{i}$ is the truncation of $T$ to $X_{i}\times X_{i}$. Then $T_{i}\in B^{p}((X_{i})_{0})$, so $T\in I_{0}$. Therefore, we have $I_{0}=B^{p}(X_{0})\cap I_{f}$. 

Next, we show that $B^{p}(X_{0})+I_{f}=B^{p}(X_{f})$. It is straightforward to check that $B^{p}(X_{0})$ and $I_{f}$ are closed ideals of $B^{p}(X_{f})$, so $B^{p}(X_{0})+I_{f}$ is an ideal of $B^{p}(X_{f})$. Conversely, for any $T\in B^{p}(X_{f})$, we write $T=T^{(1)}+T^{(2)}$, where $\supp(T^{(1)})\subseteq X_{i}\times X_{i}$ for sufficiently large $i$, and $\supp(T^{(2)})$ is $C_{0}$ controlled. Thus $T^{(1)}\in B^{p}(X_{i})\subseteq I_{f}$, $T^{(2)}\in B^{p}(X_{0})$, so $B^{p}(X_{0})+I_{f}$ is dense in $B^{p}(X_{f})$. It can be verified that $B^{p}(X_{0})+I_{f}$ is closed. Indeed, let $\{T^{(1)}_{n}+T^{(2)}_{n}\}$ be a Cauchy sequence in $B^{p}(X_{0})+I_{f}$, and put $T_{n}\triangleq T^{(1)}_{n}+T^{(2)}_{n}$, then there exists an operator $T\in B^{p}(X_{f})$ such that $T_{n}\xrightarrow{\Vert\cdot\Vert}T$. As $\Vert T^{(1)}_{n}\Vert\leq\Vert T_{n}\Vert$, we conclude that $\{T^{(1)}_{n}\}$ is a Cauchy sequence. Since $T^{(2)}_{n}=T_{n}-T^{(1)}_{n}$, we can deduce that $\{T^{(2)}_{n}\}$ is also a Cauchy sequence. Therefore, there exist $T^{(1)}\in I_{f}$ and $T^{(2)}\in B^{p}(X_{0})$ such that $T^{(1)}_{n}\xrightarrow{\Vert\cdot\Vert}T^{(1)}$ and $T^{(2)}_{n}\xrightarrow{\Vert\cdot\Vert}T^{(2)}$, namely, $T=T^{(1)}+T^{(2)}\in B^{p}(X_{0})+I_{f}$. It follows that $B^{p}(X_{0})+I_{f}=B^{p}(X_{f})$ as required. Hence, we obtain the isomorphism
$$
\frac{B^{p}(X_{0})}{I_{0}}=\frac{B^{p}(X_{0})}{B^{p}(X_{0})\cap I_{f}}\cong \frac{B^{p}(X_{0})+I_{f}}{I_{f}}=\frac{B^{p}(X_{f})}{I_{f}}.
$$

Eventually, for any bounded geometry space $W$, we have the following commutative diagram:
\begin{equation}\label{eq 3}
\begin{CD}
K_{*+1}(B^{p}(X_{0})/I_{0})@>{\cong}>>K_{*+1}(B^{p}(X_{f})/I_{f})
\\@VV{\cong}V@VVV\\
K_{*}(I_{0})@>>>K_{*}(I_{f})
\\@AA{\cong}A@AA{\cong}A\\
KX^{p}_{*}(W)@>{\mu}>> K_{*}(B^{p}(W)).
\end{CD}
\end{equation}
Thanks to Theorem \ref{th 4.7} and Theorem \ref{th 4.12}, we see that $KX^{p}_{*}(W)\cong K_{*}(I_{0})$ and $K_{*}(I_{f})\cong K_{*}(B^{p}(W))$.
By Theorem \ref{th 4.5}, we know that $K_{*}(B^{p}(X_{0}))=0$, thus $K_{*+1}(B^{p}(X_{0})/I_{0})\cong K_{*}(I_{0})$. Therefore, from diagram (\ref{eq 3}), we deduce that $\mu$ is an isomorphism if and only if $K_{*+1}(B^{p}(X_{f})/I_{f})\cong K_{*}(I_{f})$. Since $K_{*+1}(B^{p}(X_{f})/I_{f})\cong K_{*}(I_{f})$ if and only if $K_{*}(B^{p}(X_{f}))=0$, we get that $\mu$ is an isomorphism if and only if 
$K_{*}(B^{p}(X_{f}))=0$.
\end{proof}

According to Theorem \ref{th 4.8}, we can define the obstruction group for the $L^{p}$ coarse Baum-Connes conjecture as follows.
\begin{definition}[The obstruction group for the $L^{p}$ coarse Baum-Connes conjecture]
\quad Let $W$ be a bounded geometry metric space, and let $\mathcal{U}_{*}$ be an anti-\v{C}ech sequence for $W$. Let $X=X(W,\mathcal{U}_{*})$ be the corresponding total coarsening space, and let $X_{f}=X(W,\mathcal{U}_{*})_{f}$ be the fusion coarse structure on $X$. The obstruction group for the $L^{p}$ coarse Baum-Connes conjecture is the $K$-theory group $K_{*}(B^{p}(X_{f}))$.
\end{definition}

\section{An application to finite asymptotic dimension}
The notion of asymptotic dimension due to M. Gromov is a large-scale dimension of a metric space, and also a large-scale geometric invariant \cite{Gro93}. In this section, our purpose is to show that the $L^{p}$ coarse Baum-Connes conjecture holds for bounded geometry metric spaces with finite asymptotic dimension. Based on the results of the last section, we only need to prove that the obstruction groups vanish. Indeed, we can prove this conclusion. However, it is a bit difficult to calculate the obstructions directly, so we have to resort to a new coarse structure, called the hybrid coarse structure. This coarse structure defined in \cite{Wri05} is a structure between the fusion and bounded coarse structures, and is coarser and more computable than the fusion structure. Furthermore, we will show that the $K$-theory of the $L^{p}$ Roe algebra of the total coarsening space equipped with the hybrid structure is zero in such spaces, and then prove that the $K$-theory groups of $L^{p}$ Roe algebras of the fusion and hybrid coarse structures are isomorphic. Eventually, we obtain that the obstruction groups arising from the fusion structure vanish in this case.

\subsection{Asymptotic dimension}
In this subsection, we review the definition of the asymptotic dimension of a metric space. To introduce it, we need the concept of $R$-degree of an open cover for $R>0$. More formally, let $\mathcal{U}$ be an open cover of a metric space $W$. Given $R>0$, the $R$-degree of $\mathcal{U}$ is given by
$$
R\text{-}degree(\mathcal{U})=\sup\limits_{w\in W}{\text cardinality}\{U\in\mathcal{U}\mid d(w, U)<R\}.
$$
\begin{definition}\cite{Gro93}\label{def 5.2}
A metric space $W$ has asymptotic dimension at most $m$ if for all $R>0$, there exists an open cover $\mathcal{U}$ of $W$ with $\Diam(\mathcal{U})<\infty$ and with $R$-degree of $\mathcal{U}$ at most $m+1$. The asymptotic dimension of $W$ is the smallest $m$ such that $W$ has an asymptotic dimension at most $m$. It is denoted by $\asdim W=m$.
\end{definition}

We now look at some typical examples.
\begin{example}\cite{NY12}
$\asdim \mathbb{Z}^{n}=n$, $n=1,2,\cdots$.
\end{example}

\begin{example}\cite{NY12}
Let $T$ be a tree. Then $\asdim T \leq 1$.
\end{example}

\begin{example}\cite{NY12}
The asymptotic dimension of the free group $\mathbb{F}_{n}$ equals 1.
\end{example}

\begin{remark}
We say that a metric space $W$ has finite asymptotic dimension if there exists an $m\in\mathbb{N}$ such that $\asdim W\leq m$. There are plenty of metric spaces and groups with finite asymptotic dimension, such as Gromov's hyperbolic groups as metric spaces with word metric \cite{Gro93, Roe05},  Coxeter groups \cite{DJ99}, almost connected Lie groups \cite{NY12}, certain relatively hyperbolic groups \cite{Osi05}, finite-dimensional CAT(0)-cube complexes \cite{Wri12}, mapping class groups \cite{BBF15}, and all hierarchically hyperbolic groups \cite{BHS17}, and finitely generated one-relator groups \cite{Tse23}.
\end{remark}

\begin{remark}
Recall that the degree of an open cover $\mathcal{U}$ of a metric space $X$, denoted $Degree(\mathcal{U})$, is the supremum over points $x\in X$ of the number of elements $U$ of $\mathcal{U}$ containing $x$.
From Definition \ref{def 5.2}, we see that for a uniformly discrete metric space $W$ with finite asymptotic dimension at most $m$, there exists an anti-\v{C}ech sequence $\mathcal{U}_{*}$ for $W$ with $Degree(\mathcal{U}_{i})\leq m+1$ for all $i\in\mathbb{N}$.
\end{remark}

\subsection{The $L^{p}$ coarse Baum-Connes conjecture for spaces with finite asymptotic dimension} In this subsection, our destination is to give a new proof of the $L^{p}$ coarse Baum-Connes conjecture for bounded geometry metric spaces with finite asymptotic dimension using the conclusions of the previous sections.

We first review the notion of the hybrid coarse structure. This key definition that bridges the fusion coarse structure with the bounded coarse structure is as follows. 
\begin{definition}[The hybrid coarse structure]\cite{Wri05}
Let $X=X(W,\mathcal{U}_{*})$ be the total coarsening space, and let $X_{i}=X_{i}(W,\mathcal{U}_{*})$ be the partial coarsening spaces. The hybrid coarse structure on $X$, denoted $X_{h}=X(W,\mathcal{U}_{*})_{h}$, is the coarse structure for which a set $A\subseteq X\times X$ is controlled if and only if
\begin{itemize}
\item $d|_{A}$ is bounded, i.e. $A$ is controlled for the bounded coarse structure, and
\item $\sup\{d(x,y)\mid (x,y)\in A\backslash(X_{i}\times X_{i})\}\rightarrow 0$ as $i$ tends to infinity.
\end{itemize}
\end{definition}

\begin{remark}
The maps $X_{0}\rightarrow X_{f}\rightarrow X_{h}\rightarrow X$ are coarse. In other words, the fusion and hybrid coarse structures are coarse transitions from the $C_{0}$ structure to the bounded structure.
\end{remark}

The next fact is a key step in the proof of the subsequent results.
\begin{remark}
For a uniformly discrete bounded geometry metric space $X$ and a cover $\mathcal{U}$ of $X$ with $\Diam(\mathcal{U})<\infty$, we have $Degree(\mathcal{U})<\infty$ and $N_{\mathcal{U}}$ is finite dimensional with $Degree(\mathcal{U})=\dim(N_{\mathcal{U}})+1$.
\end{remark}
The following technical proposition constructs a Lipschitz map which is linearly homotopic to a collapsing map.
\begin{proposition}\label{prop 5.8}\cite{Wri05}
Let $W$ be a uniformly discrete bounded geometry metric space with asymptotic dimension at most $m$, and let $\mathcal{U}_{*}$ be an anti-\v{C}ech sequence for $W$ with degrees bounded by $m+1$. Let $X=X(W,\mathcal{U}_{*})$ be the total coarsening space of $(W,\mathcal{U}_{*})$, let $d$ be the metric on $X$, and let $\pi$ be the quotient map from $X$ to $[1,\infty)$. Then for each $i_1$ and each $R$, $\varepsilon>0$, there exist an $i_{2}>i_{1}$ and a map $\beta: \pi^{-1}[i_1,\infty)\rightarrow \pi^{-1}[i_2,\infty)$ such that
\begin{itemize}
\item $d(\beta(x),\beta(x'))\leq\varepsilon d(x,x')$ for $x, x'\in N_{\mathcal{U}_{i_1}}$ with $d(x,x')<R$;
\item $\beta$ is Lipschitz with constant 4;
\item $\beta(x)=x$ for $x\in X$ with $\pi(x)\geq i_2$;
\item if $x\in X$ with $i_1\leq\pi(x)\leq i_2$, then $\beta(x)\in N_{\mathcal{U}_{i_2}}$ and there exists a simplex $\sigma$ of $N_{\mathcal{U}_{i_2}}$ containing both $\Phi_{i_2}(x)$ and $\beta(x)$, hence 
$\Phi_{i_2}$ is linearly homotopic to $\beta$ as a map from $\pi^{-1}[i_1,\infty)\rightarrow\pi^{-1}[i_2,\infty)$.
\end{itemize}
\end{proposition}

With the help of the preceding proposition, we can now prove that under the assumption of finite asymptotic dimension, the $L^{p}$ Roe algebra of the total coarsening space with the hybrid coarse structure has trivial $K$-theory.

\begin{theorem}
Let $W$ be a uniformly discrete bounded geometry metric space with asymptotic dimension at most $m$, let $\mathcal{U}_{*}$ be an anti-\v{C}ech sequence for $W$ with degrees bounded by $m+1$, and let $X$ be the associated coarsening space. Then the $K$-theory groups $K_{*}(B^{p}(X_{h}))=0$.
\end{theorem}

\begin{proof}
This theorem can be proved in the same way as shown in [\cite{Wri05}, Theorem 5.9]. Moreover, the trick of the proof is also to use Proposition \ref{prop 5.8} to construct a sequence that satisfies the hypothesis of Lemma \ref{lemma 3.11}. More precisely, we construct a desired sequence as follows:

Let $i_{1}=1$. Using Proposition \ref{prop 5.8}, we choose $i_{2}>i_{1}=1$, and a map $\beta_{1}: X=\pi^{-1}[i_{1},\infty)\rightarrow \pi^{-1}[i_{2},\infty)$ with $d(\beta_{1}(x),\beta_{1}(x'))\leq d(x,x')$ for $x,x'$ with $d(x,x')<1$. Inductively, we apply Proposition \ref{prop 5.8} to construct maps $\beta_{j}: \pi^{-1}[i_{j},\infty)\rightarrow \pi^{-1}[i_{j+1},\infty)$ with $d(\beta_{j}(x),\beta_{j}(x'))<\frac{1}{j}d(x,x')$ for $x,x'$ with $d(x,x')<j$. Next, we define
$$
\alpha_{i_{j}}=\Phi_{i_{j}}\circ\beta_{j-2}\circ\cdots\circ\beta_{1}: X\rightarrow \pi^{-1}[i_{j},\infty).
$$
By Proposition \ref{prop 5.8}, we let $\gamma_{j,t}$ denote the linear homotopy from $\Phi_{i_{j}}$ to $\beta_{j-1}$. For $t\in [0,1]$, we set
$$
\alpha_{t}=\Phi_{t}\circ\gamma_{j,t}\circ\beta_{j-2}\circ\cdots\circ\beta_{1}: X\rightarrow\pi^{-1}[t,\infty), \quad t\in [i_{j},i_{j+1}].
$$
Then $\alpha_{t}$ is a homotopy between $\alpha_{i_{j}}$ and $\alpha_{i_{j+1}}$. Put $\{t_{k}\}$ be an increasing sequence tending to infinity with $t_{0}=1$, $t_{k+1}-t_{k}\rightarrow 0$ when $k$ tends to infinity, and let $\{i_{j}\}$ be an integer subsequence of $\{t_{k}\}$. Similarly, we can verify that the sequence $\{\alpha_{t_{k}}\}$ is properly supported, uniformly controlled, and has uniformly close steps. The remainder of the argument is analogous to that of Theorem 5.9 in \cite{Wri05} and is left to the readers.
\end{proof}

Below, we identify the $K$-theory of the $L^{p}$ Roe algebra of the total coarsening space equipped with the fusion structure with the $K$-theory of the $L^{p}$ Roe algebra of the total coarsening space equipped with the hybrid structure, and the latter is shown to be trivial by the above theorem. 

\begin{theorem}\label{th 5.10}
Let $W$ be a uniformly discrete bounded geometry metric space with asymptotic dimension at most $m$, and let $\mathcal{U}_{*}$ be an anti-\v{C}ech sequence for $W$ with degrees bounded by $m+1$. Then the forgetful map $B^{p}(X(W,\mathcal{U}_{*})_{f})\rightarrow B^{p}(X(W,\mathcal{U}_{*})_{h})$ induces an isomorphism on $K$-theory. In particular, $K_{*}(B^{p}(X(W,\mathcal{U}_{*})_{f}))=0$.
\end{theorem}

\begin{proof}
Since the coarsening space $X$ is an adjunction space formed by products $N_{\mathcal{U}_{i}}\times[i,i+1]$ with $\dim N_{\mathcal{U}_{i}}\leq m$ for each $i\in\mathbb{N}$, we subdivide products to make $X$ a simplicial complex of dimension $m+1$. Equip $X$ with the uniform spherical metric. Our task now is to show that the forgetful map between the fusion and hybrid structures induces an isomorphism for any finite-dimensional simplicial complex. Let $X'$ be a metric simplicial complex of dimension $m'$ equipped with a map $\pi':X'\rightarrow\mathbb{R}^{+}$, and suppose that
\begin{itemize}
\item either $m'=0$ and $X'$ is uniformly discrete;
\item or $m'>0$, $X'$ is connected and the metric on $X'$ is the uniform spherical metric.
\end{itemize}
We proceed by induction on the dimension $m'$. In the case $m'=0$, for any hybrid controlled operator $T$ on a uniformly discrete space, we write $T=T_{1}+T_{2}$, where $\supp(T_{1})\subseteq\pi'^{-1}[0, i]\times\pi'^{-1}[0, i]$ for some integer $i$ and $\prop(T_{2})=0$, thus $T$ is fusion controlled. Since a fusion controlled operator is hybrid controlled for any space, the notions of controlled operators agree. Hence, $L^{p}$ Roe algebras for the hybrid and fusion structures are isomorphic, whence the result holds for $m'=0$. For the general case, proceed by induction on $m'$.

For $m'>0$, we let $X'=Y_{0}\cup\cdots\cup Y_{m'}$, where $Y_{k}$ is the union of stars about the barycenters of $k$-simplices in the second barycentric subdivision $X^{(2)}$ of $X'$, and let $\widetilde{Y}_{k}$ be the union of $Y_{k}$ and the 1-skeleton of $X^{(2)}$. Further, we set $G_{k}$ to be the graph consisting of those edges of $X^{(2)}$ which are not contained in $Y_{k}$. Let $G_{k}=V_{k}\cup E_{k}$, where $V_{k}$ is the graph consisting of uniformly separated stars about the vertices in $G_{k}$ and $E_{k}$ is the graph consisting of uniformly separated segments in the edges of $G_{k}$.

We build the binary decomposition as follows:
\begin{align*}
 X &=Z_{m} (\text{ the root })\\
 Z_{k} &=Z_{k-1}\cup\widetilde{Y}_{k}\\
Z_{0} &=\widetilde{Y}_{0}\\
\widetilde{Y}_{k} &=Y_{k}\cup G_{k}\\
G_{k} &=V_{k}\cup E_{k}.
\end{align*}
Below, we show that the decomposition is admissible for the forgetful transformation $K_{*}(B^{p}(X'_{f}))\rightarrow K_{*}(B^{p}(X'_{h}))$. First, we show that each fork is excisive. Since $X'$ is connected, $Z_{k}$, $\widetilde{Y}_{k}$ and $G_{k}$ are connected. By Lemma \ref{lemma A.8}, the metric on each fork is bi-Lipschitz equivalent to the uniform spherical metric. It follows that the decompositions are all excisive for both the hybrid and fusion structures. 

Next, we show that at each fork, the intersection leads to an isomorphism.
Observe that $Y_{k}\cap G_{k}$ and $V_{k}\cap E_{k}$ are 0-dimensional, thus they give rise to isomorphisms on $K$-theory. Evidently, $\dim(Z_{k-1}\cap\widetilde{Y}_{k})=m'-1$. If $m'=1$, $Z_{k-1}\cap\widetilde{Y}_{k}$ is uniformly discrete. If $m'>1$, by Lemma \ref{lemma A.8}, its metric is bi-Lipschitz equivalent to the uniform spherical metric. By inductive hypothesis, the intersection $Z_{k-1}\cap\widetilde{Y}_{k}$ yields an isomorphism.

Finally, we illustrate that each leaf gives an isomorphism. Observe that each leaf is coarsely equivalent to a uniformly discrete space for both the hybrid and fusion structures, so the transformation induces an isomorphism. Therefore, we conclude that the binary decomposition is admissible when $m'>0$. By Lemma \ref{lemma 3.16}, we thus derive the required isomorphism. Let $X'=X$, $\pi'=\pi$, and $m'=m+1$, we obtain the desired result.
\end{proof}

Combining Theorem \ref{th 4.8} with Theorem \ref{th 5.10}, we give a new proof of the following $L^{p}$ coarse Baum-Connes conjecture for spaces with finite asymptotic dimension, which was first established by Zhang and Zhou \cite{ZZ21}.
\begin{theorem}\label{th 5.3}
Let $W$ be a bounded geometry metric space with finite asymptotic dimension. For $p\in[1,\infty)$, the $L^{p}$ coarse Baum-Connes assembly map $\mu: KX^{p}_{*}(W)\rightarrow K_{*}(B^{p}(W))$ is an isomorphism.
\end{theorem}

\begin{proof}
Let $\asdim W=m$ and $\mathcal{U}_{*}$ be an anti-\v{C}ech sequence for $W$ with $\dim N_{\mathcal{U}_{i}}\leq m$. Let $X$ be the total coarsening space. We have the following commutative diagram: 
\begin{equation}\label{eq 4}
\begin{CD}
K_{*+1}(B^{p}(X_{0})/I_{0})@>{\cong}>>K_{*+1}(B^{p}(X_{f})/I_{f})
\\@VV{\cong}V@VVV\\
K_{*}(I_{0})@>>>K_{*}(I_{f})
\\@AA{\cong}A@AA{\cong}A\\
KX^{p}_{*}(W)@>{\mu}>> K_{*}(B^{p}(W)).
\end{CD}
\end{equation}
Thanks to Theorem \ref{th 5.10}, we know that $K_{*}(B^{p}(X_{f}))=0$, thus $K_{*+1}(B^{p}(X_{f})/I_{f})\cong K_{*}(I_{f})$. Hence, from diagram (\ref{eq 4}), $\mu$ is an isomorphism.
\end{proof}

\section*{Acknowledgments}
The authors are grateful for the valuable comments and suggestions from the referees, which greatly enhanced the paper's clarity.

\end{document}